\documentclass[a4paper,oneside,11pt]{article}%
\usepackage{amsmath}
\usepackage{cite}
\usepackage[utf8]{inputenc}
\usepackage{amsfonts}
\usepackage{indentfirst}
\usepackage{amssymb}
\usepackage{graphicx}%
\usepackage{xcolor}
\usepackage{appendix}
\usepackage{amsthm}
\usepackage[colorlinks,citecolor=blue,urlcolor=blue]{hyperref}

\providecommand{\U}[1]{\protect \rule{.1in}{.1in}}

\newtheorem{theorem}{Theorem}[section]

\newtheorem{assumption}[theorem]{Assumption}
\newtheorem{example}[theorem]{Example}

\newtheorem{lemma}[theorem]{Lemma}

\newtheorem{remark}[theorem]{Remark}

\numberwithin{equation}{section}

\usepackage{hyperref}
\begin{document}

\title{Stochastic maximum principle for recursive optimal control problems with varying terminal time}
\author{Jiaqi Wang\thanks{Shandong University-Zhong Tai Securities Institute for Financial Studies, Shandong University, PR China.} \quad Shuzhen Yang\thanks{Shandong University-Zhong Tai Securities Institute for Financial Studies, Shandong University, PR China, (yangsz@sdu.edu.cn). This work was supported by the National Key R\&D program of China (Grant No.2018YFA0703900, ZR2019ZD41), National Natural Science Foundation of China (Grant No.11701330), and Taishan Scholar Talent Project Youth Project.}
}
\date{}
\maketitle

\textbf{Abstract}: This paper introduces a new recursive stochastic optimal control problem driven by a forward-backward stochastic differential equations (FBSDEs), where the terminal time varies according to the constraints of the state of the forward equation. This new optimal control problem can be used to describe the investment portfolio problems with the varying investment period. Based on novel $\rho$-moving variational and adjoint equations, we establish the stochastic maximum principle for this optimal control problem including the classical optimal control problem as a particular case. Furthermore, we propose an example to verify our main results.



\addcontentsline{toc}{section}{\hspace*{1.8em}Abstract}

\section{Introduction}
In the classical recursive stochastic control problem driven by FBSDEs, the cost functional is given as follows:
\begin{equation}\label{1.1}
\mathit{J\left ( u\left ( \cdot  \right )  \right ) }
=\mathbb{E} \left [\int_{0}^{T}
l\left ( X^u\left ( t  \right ) ,Y^u\left ( t \right ),Z^u\left ( t \right ),u\left ( t \right )
,t    \right )dt+\beta\left ( X^u\left (T \right )  \right ) +
\gamma \left ( Y^u\left ( 0 \right )  \right )    \right ],
\end{equation}
where
$l(\cdot)$ denotes the running cost, $\beta(\cdot)$ and $\gamma(\cdot)$ are the initial and terminal cost. The tuple $\left ( X^u\left ( t  \right ) ,Y^u\left ( t \right ),Z^u\left ( t \right )\right)$ satisfy the following controlled FBSDEs:
\begin{equation*}
\begin{cases}
 X^u\left ( t \right )=X_0 + \int_{0}^{t} f\left ( s,X^u\left ( s \right ),  u\left ( s \right ) \right )ds+ \int_{0}^{t}\sigma \left ( s,X^u\left ( s \right ),  u\left ( s \right ) \right )dW(s)
 \\Y^u\left ( t \right ) = \Psi \left ( X^u\left ( T \right )\right ) -\int_{t}^{T}g\left ( s,X^u\left ( s\right ), Y^u\left ( s \right ) ,Z^u\left ( s \right ) ,u\left ( s \right )   \right )ds -\int_{t}^{T}Z^u\left ( s \right ) dW(s),
\end{cases}
\end{equation*}
where $f$, $\sigma$, $\Psi$ and $g$ are given functions.

Well-known that the stochastic maximum principle is a powerful tool for solving stochastic optimal control problem, and many researchers have made contributions in this fields. We refer readers to \cite{peng1993backward} for the new form of stochastic maximum principle of FBSDEs which is the theoretical foundation of the recursive optimal control problem. Further more see \cite{peng1999fully,tao2012maximum,wang2009maximum}.  For a class of stochastic partial differential equations controlled through the boundary,
\cite{guatteri2011stochastic} proposed the necessary conditions for solving the stochastic optimal control problem. \cite{wu2003fbsde} studied the linear quadratic stochastic optimal control problem with random jumps. \cite{wang2013optimal} considered an optimal control problem where the control is adapted to a sub-filtration. \cite{han2010maximum} investigated the optimal control problems for backward doubly stochastic control systems. Furthermore, we refer readers to \cite{tang2014variational, li2012stochastic} for more regarding theories of stochastic maximum principles.

There are many researches considered the optimal control problems under state constraints. \cite{wei2016stochastic} considered an optimal control problem with state constraints under a mean-field FBSDEs. \cite{ji2013maximum} studied the optimal control problem with terminal constraints where the system is driven by a fully coupled FBSDEs. \cite{huang2020necessary} derived the necessary condition for the existence of the optimal control under FBSDEs with Levy process and established the maximum principle with respect to some initial and terminal state constraints. Furthermore, we refer readers to \cite{ji2012maximum, liangquan2020necessary, moon2022linear, arada2000optimal} for more regarding researches on optimal control under state constraints.

\cite{yang2020varying} introduced a novel varying terminal time optimal control problem where the terminal time varies according to control $u(\cdot)$ which can improve the performance of the classical cost functional under stochastic differential equations. We refer readers to \cite{yang2022varying} for the application of the varying terminal time optimal control problem in mean-variance investment model. Indeed, a varying terminal investment time can reduce variance by cost functional (\ref{1.1}). In this present paper, we extend the varying terminal time optimal control problem to a recursive varying terminal optimal control problem, that is, a stopping criterion for the terminal time of FBSDEs is described as follows:
$$
 \tau ^u=inf\left \{ t:\mathbb{E}\left [ \Phi \left ( X^u\left ( t
 \right )  \right )  \right ] \ge \alpha , t\in\left [ 0,\mathit{T}  \right ]  \right \} \wedge \mathit{T}.
$$
Here, $\Phi$ is a given function used to describe the constraints on state $X^u$. The target is to minimize the following cost functional under the varying terminal time $\tau^u$:
\begin{equation}\label{11cost}
\mathit{J\left ( u\left ( \cdot  \right )  \right ) }
=\mathbb{E} \left [\int_{0}^{\tau^u}
l\left ( X^u\left ( t  \right ) ,Y^u\left ( t \right ),Z^u\left ( t \right ),u\left ( t \right )
,t    \right )dt+\beta\left ( X^u\left (\tau^u \right )  \right ) +
\gamma \left ( Y^u\left ( 0 \right )  \right )    \right ],
\end{equation}
where $\left ( X^u\left ( t  \right ) ,Y^u\left ( t \right ),Z^u\left ( t \right )\right)$ satisfies a new form of FBSDEs with a varying terminal time (See equation (\ref{2.4})). We denote the optimal pair for cost functional (\ref{11cost}) by $(\bar{u}(\cdot), \tau^{\bar{u}})$. Based on three kinds of case of optimal terminal time $\tau^{\bar{u}}$:

$(i)$ $\tau^{\bar{u}}<T$,

$(ii)$  $\inf\left \{ t:\mathbb{E}\left [ \Phi \left ( X^u\left ( t
 \right )  \right )  \right ] \ge \alpha , t\in\left [ 0,\mathit{T}  \right ]  \right \} = T$,

 $(iii)$ $\left \{ t:\mathbb{E}\left [ \Phi \left ( X^u\left ( t
 \right )  \right )  \right ]\notag \right. \left. \ge \alpha, t\in\left [ 0,\mathit{T}  \right ]  \right \} = \emptyset$,

 \noindent we establish a novel stochastic maximum principle providing a necessary condition for the above new optimal control problem.

 The main contributions of this study are twofold:

(\textbf{i}) A novel recursive stochastic optimal control problem is proposed, which can be used to describe the varying terminal time investment problem.

(\textbf{ii}) Based on novel $\rho$-moving variational and adjoint equations, we establish the stochastic maximum principe for this optimal control problem including the classical optimal control problem as a particular case.

The reminder of the paper is organized as follows. In section \ref{section2}, we introduce a varying terminal stochastic optimal control problem driven by FBSDEs and present the main results of this paper. Then, we present some preliminaries results for the proof of the stochastic maximum principle in section \ref{sectionIII}. We prove the stochastic maximum principle and construct an example to illustrate the main results in section \ref{sectionIV}. In section \ref{sectionV}, we conclude this paper.

\section{Varying terminal time optimal control problem}\label{section2}

In this section, we introduce a new stochastic optimal control problem driven by a FBSDEs with a varying terminal time. The objective of this study is to derive a stochastic maximum principle for this new optimal control problem. Let $\mathit{W}$ be a d-dimensional standard Brownian motion defined on a complete filtered probability space $\left ( \Omega ,\ \left \{ \mathcal{F}_t  \right \}_{t\ge 0},\ \mathit{P} \right ) $, where $\left \{ \mathcal{F}_t  \right \}_{t\ge 0}$ is the argumentation of the natural filtration generated by the Brownian motion $\mathit{W}$.
Let $\mathit{U}$ be a nonempty convex subset of $\mathbb{R} ^k$ and  $\mathcal{U}[0,T]$ be
$$
\mathcal{U}[0,T]=\left \{ u\left ( \cdot  \right ) \in\mathcal{M}^2\left ( \mathbb{R}^k  \right )\ | \ u\left ( t
 \right )\in \mathit{U},\ 0\leq t\leq T,\ a.s.    \right \}.
$$
where $\mathit{T}>0$ is a given constant.

For any given admissible control $u\left ( \cdot  \right ) \in \mathcal{U}[0,T]$, we introduce the following varying terminal time controlled FBSDEs:
\begin{equation}
\label{2.4}
\begin{cases}
 dX^u\left ( t \right )=f\left ( t,X^u\left ( t \right ),  u\left ( t \right ) \right )dt+\sigma \left ( t,X^u\left ( t \right ),  u\left ( t \right ) \right )dW(t),
 \\X\left ( 0 \right ) = x_0,
 \\dY^u\left ( t \right ) = g\left ( t,X^u\left ( t \right ), Y^u\left ( t \right ) ,Z^u\left ( t \right ) ,u\left ( t \right )   \right )dt +Z^u\left ( t \right ) dW(t),
 \\Y\left ( \tau^u \right ) = \Psi \left ( X^u\left ( \tau^u \right )  \right),
\end{cases}
\end{equation}
where $x_0$ and function $\Psi \left ( \cdot  \right ) $ are given, and terminal time $\tau^u$ satisfies
\begin{equation}\label{varying_time}
    \tau ^u=inf\left \{ t:\mathbb{E}\left [ \Phi \left ( X^u\left ( t
 \right )  \right )  \right ] \ge \alpha , t\in\left [ 0,\mathit{T}  \right ]  \right \} \wedge \mathit{T}.
\end{equation}
In this study, we focus on a varying terminal time which depends on the forward state $X^u(\cdot)$. The related cost functional is given as follows:
\begin{equation}\label{cost_functional}
    \mathit{J\left ( u\left ( \cdot  \right )  \right ) }
=\mathbb{E} \left [\int_{0}^{\tau ^u}
l\left ( X^u\left ( t  \right ) ,Y^u\left ( t \right ),Z^u\left ( t \right ),u\left ( t \right )
,t    \right )dt+\beta \left ( X^u\left ( \tau ^u \right )  \right ) +
\gamma \left ( Y^u\left ( 0 \right )  \right )    \right ].
\end{equation}

Let $\mathit{f},\ \sigma,\ \mathit{g},\ \Psi,\ \mathit{l},\ \beta,\ \gamma,\ \Phi$ be such that:
\begin{align*}
&\mathit{f}\left (\omega, t,x,u \right ) :\Omega\times \left [ 0,T \right ] \times \mathbb{R}^n \times \mathit{U} \to \mathbb{R}^n,\\
&\sigma\left (\omega,  t,x,u \right ) :\Omega\times\left [ 0,T \right ] \times \mathbb{R}^n \times \mathit{U} \to \mathbb{R}^{n \times d},\\
&\mathit{g}\left (\omega,  t,x,y,z,u \right ) :\Omega\times\left [ 0,T \right ] \times \mathbb{R}^n \times \mathbb{R}^m  \times \mathbb{R}^{m \times d} \times \mathit{U} \to \mathbb{R}^m,\\
&\mathit{l} \left (\omega, x,y,z,u,t \right ) : \Omega\times\mathbb{R}^n \times \mathbb{R}^m
\times \mathbb{R}^{m \times d} \times \mathit{U} \times \left [ 0,T \right ] \to \mathbb{R},\\
&\gamma\left (\omega,  x \right ) : \Omega\times\mathbb{R}^m \to \mathbb{R},\\
&\Phi\left (\omega,  x \right ),\ \beta \left (\omega, x \right ) : \Omega\times\mathbb{R}^n \to \mathbb{R},\\
&\Psi\left (\omega,  x \right ) :\Omega\times\mathbb{R}^n \to \mathbb{R}^m.
\end{align*}
For notations simplicity, we omit $(t,\omega)$ in the above functions. Furthermore, we assume that $\mathit{f},\ \sigma,\ \mathit{g},\ \Psi,\ \mathit{l},\ \beta,\ \gamma$, and $\Phi$ satisfy the following assumptions.

\begin{assumption}\label{assumption_1}
$\mathit{f},\ \sigma,\ \mathit{g},\ \Psi,\ \mathit{l},\ \beta ,\ \gamma,$ and $\Phi$ are continuous in $\mathbb{R}^n \times \mathbb{R}^m
\times \mathbb{R}^{m \times d} \times \mathit{U} \times \left [ 0,T \right ]$, and continuously differentiable with respect to $\left ( x,y,z,u,t \right ) $.
\end{assumption}

\begin{assumption}\label{assumption_2}
The derivatives of $f,\ \sigma,\ g,\ \Psi$ are bounded. The derivatives of $\mathit{l} $ are bounded by $C\left ( 1+\left | x \right |+\left | y \right |+\left | z \right | +\left | u \right |    \right ) $ and the derivatives of $\beta $, $\Psi$ and $\gamma$ with respect to $x$ are bounded by $C\left ( 1+\left | x \right |  \right ) $.
\end{assumption}

\begin{assumption}\label{assumption_3}
Let $\Psi$, $\beta$ be twice differentiable at x, with its derivatives in x be uniformly continuous in x. Let  $\Phi$ be three-times differentiable at x, with its derivatives in x be uniformly continuous in x.
\end{assumption}
\begin{remark}
Since $x_0$, $\Psi$ and $\tau^u$ are deterministic, based on the results in Peng\cite{peng1993backward}, it follows from Assumptions \ref{assumption_1}, \ref{assumption_2} that there exists a unique triple
$$\left ( X^u\left ( \cdot  \right ),Y^u\left ( \cdot  \right ),Z^u \left ( \cdot  \right ) \right ) \in \mathcal{M}^2\left ( \mathbb{R}^n  \right ) \times \mathcal{M}^2\left ( \mathbb{R}^m  \right )
\times \mathcal{M}^2\left ( \mathbb{R}^{m\times d}  \right )$$
satisfying equation (\ref{2.4}). Assumption \ref{assumption_3} is a necessary condition to derive the variation of $\tau^u$ introduced in Yang \cite{yang2020varying}.
\end{remark}
The corresponded solution $\left ( X^u\left ( \cdot  \right ),Y^u\left ( \cdot  \right ),Z^u \left ( \cdot  \right ) \right )$ is called the state variable or trajectory under control $u(\cdot)$. The stochastic optimal control problem is to minimize the cost functional (\ref{cost_functional}) over admissible controls $\mathcal{U}[0,\tau ^u ]$, and the control $\bar{u}\left ( \cdot  \right )$  satisfying
\begin{equation}
\label{min_cost_functional}
    J\left ( \bar{u}\left ( \cdot  \right )   \right )
=\underset{u\left ( \cdot  \right ) \in \mathcal{U}\left [ 0,\tau ^u \right ]  }{inf}
J\left ( u\left ( \cdot  \right )   \right )
\end{equation}
is called an optimal control. The corresponded state trajectory triple $\left ( X^{\bar{u} }\left ( \cdot  \right ),Y^{\bar{u} }\left ( \cdot  \right ),Z^{\bar{u} }\left ( \cdot  \right ) \right )$ is called an optimal state trajectory and $\tau^{\bar{u}}$ is called the optimal terminal time. The main results of this paper is the following stochastic maximum principle.
\begin{theorem}\label{Theorem}
Let Assumptions \ref{assumption_1}, \ref{assumption_2} and \ref{assumption_3} hold.  Let $h^{\bar{u}}\left(\tau^{\bar{u}}\right) \neq 0$ and $h^{\bar{u}}(\cdot)$ be continuous at the point $\tau^{\bar{u}}$, where $h^{{u}}(\cdot)$ is given in Lemma \ref{lemma 3.2}. Then, there exists $(p(\cdot),k(\cdot), q(\cdot))$ satisfying adjoint equations (\ref{hami-1}), (\ref{hami-3}) and the following hold.

(i). If $\tau^{\bar{u}}<T$, we have
\begin{equation*}
\begin{aligned}
&H_u(X^{\bar{u}}(t),Y^{\bar{u}}(t),Z^{\bar{u}}(t),\bar{u}(t),p(t),k(t),q(t),t)(u-\bar{u}(t))
+\mathcal{L}(t)\frac{\bar{h}(u-\bar{u}(t), t)}{h^{\bar{u}}\left(\tau^{\bar{u}}\right)}\ge 0
\end{aligned}
\end{equation*}
where
\begin{equation*}
\begin{aligned}
H(x, y, z, u, p, k, q, t)= &(p, f(x, u))+(k, \sigma(x, u))+(q, g(x, y, z, u, t))+l(x, y, z, u, t)\\
\mathcal{L}(t)=&q(\tau^{\bar{u}})\tilde{\Psi}^{\bar{u}}\left(\tau^{\bar{u}}\right)-q(\tau^{\bar{u}})g\left(X^{\bar{u}}\left(\tau^{\bar{u}}\right), Y^{\bar{u}}\left(\tau^{\bar{u}}\right), Z^{\bar{u}}\left(\tau^{\bar{u}}\right), \bar{u}\left(\tau^{\bar{u}}\right)\right)\\
&-\tilde{\beta}^{\bar{u}}\left(\tau^{\bar{u}}\right)
-\mathbb{E}\left[l\left(X^{\bar{u}}\left(\tau^{\bar{u}}\right), Y^{\bar{u}}\left(\tau^{\bar{u}}\right), Z^{\bar{u}}\left(\tau^{\bar{u}}\right), \bar{u}\left(\tau^{\bar{u}}\right), \tau^{\bar{u}}\right)\right],
\end{aligned}
\end{equation*}
$\tilde{\Psi}^{\bar{u}}$ and $\tilde{\beta}^{\bar{u}}$ are defined in equations (\ref{Psi}) and (\ref{beta}),
for any $u \in \mathcal{U} $, a.e. $t \in\left[0, \tau^{\bar{u}}\right), P-$a.s..

(ii). If $\inf\left\{t: \mathbb{E}[\Phi(X^{\bar{u}}(t))] \geq \alpha, t \in[0, T]\right\}=T$, we have
\begin{equation*}
\begin{aligned}
&H_u(X^{\bar{u}}(t),Y^{\bar{u}}(t),Z^{\bar{u}}(t),\bar{u}(t),p(t),k(t),q(t),t)(u-\bar{u}(t))
+\mathcal{L}(t)\frac{\bar{h}(u-\bar{u}(t), t)}{h^{\bar{u}}\left(\tau^{\bar{u}}\right)}\ge 0
\end{aligned}
\end{equation*}
for any $u \in \mathcal{U} $, a.e. $t \in\left[0, \tau^{\bar{u}}\right), P-$a.s.
or
\begin{equation*}
H_u(X^{\bar{u}}(t),Y^{\bar{u}}(t),Z^{\bar{u}}(t),\bar{u}(t),p(t),k(t),q(t),t)(u(t)-\bar{u}(t))\ge 0
\end{equation*}
for any $u \in \mathcal{U} $, a.e. $t \in\left[0, \tau^{\bar{u}}\right), P-$a.s..

(iii). If $\left\{t: \mathbb{E}[\Phi(X^{\bar{u}}(t))] \geq \alpha, t \in[0, T]\right\}=\emptyset$, we have
\begin{equation*}
H_u(X^{\bar{u}}(t),Y^{\bar{u}}(t),Z^{\bar{u}}(t),\bar{u}(t),p(t),k(t),q(t),t)(u(t)-\bar{u}(t)) \ge 0
\end{equation*}
for any $u \in \mathcal{U} $, a.e. $t \in\left[0, \tau^{\bar{u}}\right), P-$a.s..
\end{theorem}

\begin{remark}
In Theorem \ref{Theorem}, we establish the maximum principle for the recursive optimal control problem with a varying terminal time, in which we calculate the variation of $\tau^{\bar{u}}$. Note that when the varying terminal time $\tau^{{u}}$ does not depend on the control $u$, our new optimal control problem becomes the one developed in \cite{peng1993backward}. Thus, Theorem \ref{Theorem} includes the maximum principle given in \cite{peng1993backward} as a particular case, where $\frac{\bar{h}(u-\bar{u}(t), t)}{h^{\bar{u}}\left(\tau^{\bar{u}}\right)}=0$.
\end{remark}

\section{Preliminary results}\label{sectionIII}
In this section, we introduce some preliminary results which are used to prove Theorem \ref{Theorem}. Note that $U$ is a convex set. Let $\bar{u}\left ( \cdot  \right )$ be a given optimal control satisfying equation (\ref{min_cost_functional}). Suppose $0<\rho<1$ and let $v(\cdot) + \bar{u}(\cdot) \in \mathcal{U} \left [ 0,T \right ].$ We define
$$u_{\rho}\left ( t \right )=\bar{u}\left ( t \right ) + \rho v\left ( t \right ) =
\left ( 1-\rho \right ) \bar{u}\left ( t \right ) + \rho \left ( v\left ( t \right )+  \bar{u} \left ( t \right ) \right ),\  t \in \left [ 0,T \right ].$$
Clearly, $u_{\rho}\left ( \cdot  \right ) \in \mathcal{U} \left [ 0,T \right ] $ and $\left ( X^{u_\rho}\left ( \cdot \right ), Y^{u_\rho}\left ( \cdot \right ),Z^{u_\rho}\left ( \cdot \right )  \right ) $ is the solution of equation (\ref{2.4}) under control $u_\rho\left ( \cdot  \right ) $.

We introduce two lemmas which are useful in the proof of Theorem \ref{Theorem}. The following Lemma comes from Yang\cite{yang2020varying} shown that $\left | \tau^{u_\rho} -\tau^{\bar{u}} \right | \to 0$ as $\rho \to 0$ under certain continuity conditions.
\begin{lemma}\label{lemma 3.2}
Let Assumptions \ref{assumption_1}, \ref{assumption_2} and \ref{assumption_3} hold.  $\tau^{\bar{u}}$ is defined in equation (\ref{varying_time}), $h^{\bar{u} }\left ( \tau^{\bar{u} } \right ) \ne 0$ and $h^{\bar{u} }\left ( \cdot  \right )$ is continuous at point $\tau^{\bar{u}}$, where
$$
h^u\left ( t \right ) =\mathbb{E}\left [ \Phi_x\left ( X^u\left ( t \right )\right )^\top
f\left ( X^u\left ( t \right ), u\left ( t \right ) \right )
+\frac{1}{2} {\textstyle \sum_{j=1}^{d} \sigma^j \left ( X^u\left ( t \right ), u\left ( t \right ) \right )^\top\Phi_{xx}\left ( X^u\left ( t \right )\right )\sigma^j \left ( X^u\left ( t \right )
u\left ( t \right ) \right )}  \right ],
$$
and $t \in [0,T]$. Therefore, we have that
\begin{equation*}
    \lim_{\rho \to 0}\left | \tau^{\bar{u}} -\tau^{u_\rho} \right | =0.
\end{equation*}
\end{lemma}

The following Lemma in Yang\cite{yang2020varying} developed the convergence results of $\frac{\tau^{\bar{u}}-\tau^{u_\rho}}{\rho}$ when $\rho \rightarrow 0$. For notation simplicity, we set $\bar{h}(v(t), t)=\lim _{\rho \rightarrow 0} \frac{h^{u_\rho}(t)-h^{\bar{u}}(t)}{\rho}$.

\begin{lemma}\label{lemma 3.4}
Let Assumptions \ref{assumption_1}, \ref{assumption_2} and \ref{assumption_3} hold. $\tau^{\bar{u}}$ is defined in equation (\ref{varying_time}), $h^{\bar{u} }\left ( \tau^{\bar{u} } \right ) \ne 0$ and $h^{\bar{u} }\left ( \cdot  \right )$ is continuous at point $\tau^{\bar{u}}$. We have the following results.

\noindent(i). If $\tau^{\bar{u}}<T$, one obtains
\begin{equation*}\label{3.4}
    \lim _{\rho \rightarrow 0} \frac{\tau^{\bar{u}}-\tau^{u_\rho}}{\rho}=\int_0^{\tau^{\bar{u}}} \frac{\bar{h}(v(t), t)}{h^{\bar{u}}\left(\tau^{\bar{u}}\right)} dt .
\end{equation*}
(ii). If $\inf \{t: \mathbb{E}[\Phi(X^{\bar{u}}(t))] \geq \alpha, t \in[0, T]\}=T$, then there exists sequence $\rho_n \rightarrow 0$ as $n \rightarrow+\infty$ such that
\begin{equation*}
    \lim _{n \rightarrow+\infty} \frac{\tau^{\bar{u}}-\tau^{u_{\rho_n}}}{\rho_n}=\int_0^{\tau^{\bar{u}}} \frac{\bar{h}(v(t), t)}{h^{\bar{u}}\left(\tau^{\bar{u}}\right)} dt \text { or } 0 .
\end{equation*}
(iii). If $\{t: \mathbb{E}[\Phi(X^{\bar{u}}(t))] \geq \alpha, t \in[0, T]\}=\varnothing$, we have
$$
\lim _{\rho \rightarrow 0} \frac{\tau^{\bar{u}}-\tau^{u_\rho}}{\rho}=0 .
$$
\end{lemma}

In the following, let $\xi(t)$ be the solution of the following variational equation:
\begin{equation}\label{equation1}
\begin{cases}
\begin{aligned}
d\xi(t)=&\left(f_x(x(t), u(t)) \xi(t)+f_v(x(t), u(t)) v(t)\right) dt\\
&+\left(\sigma_x(x(t), u(t)) \xi(t)+\sigma_v(x(t), u(t)) v(t)\right) dW(t),\\
\xi(0)=&0.
\end{aligned}
\end{cases}
\end{equation}
Let $\left(\eta^{\rho}\left(\cdot \right), \zeta^{\rho}\left(\cdot \right) \right)$ be the solution of the following backward equation
\begin{equation}\label{equation2}
\begin{cases}
\begin{aligned}
d \eta^{\rho}(t)=&\left(g_x(x(t), y(t), z(t), u(t)) \xi(t)+g_y(x(t), y(t), z(t), u(t)) \eta^{\rho}(t)\right.\\
&\left.+g_z(x(t), y(t), z(t), u(t)) \zeta^{\rho}(t)+g_v(x(t), y(t), z(t), u(t)) v(t)\right) dt \\
&+\zeta^{\rho}(t) d W(t),\\
\eta^{\rho}\left ( \tau^{u_\rho}\wedge\right.& \left.\tau^{\bar{u}}\right )
= \kappa^{\rho}(\tau^{\bar{u}}\wedge\tau^{u_\rho}),
\end{aligned}
\end{cases}
\end{equation}
where $\tau^{u_\rho}$, $\tau^{\bar{u}}$ satisfy equation (\ref{varying_time}) and $\kappa(t)$ is given as follows:

\noindent(i). If $\tau^{\bar{u}}<T$, we set
\begin{equation}
\label{condition1}
\begin{aligned}
\kappa^{\rho}\left(\tau^{\bar{u}} \wedge\tau^{u_\rho} \right)=&\tilde{\Psi}^{\bar{u}}(\tau^{\bar{u}} \wedge\tau^{u_\rho})\frac{\tau^{u_\rho}-\tau^{\bar{u}}}{\rho}+\Psi_x\left(X^{\bar{u}}\left(\tau^{\bar{u}} \wedge\tau^{u_\rho}\right)\right)^\top \xi\left(\tau^{\bar{u}} \wedge\tau^{u_\rho}\right)\\
&+g\left(X^{\bar{u}}\left(\tau^{\bar{u}} \wedge\tau^{u_\rho}\right), Y^{\bar{u}}\left(\tau^{\bar{u}} \wedge\tau^{u_\rho}\right), Z^{\bar{u}}\left(\tau^{\bar{u}} \wedge\tau^{u_\rho}\right), \bar{u}\left(\tau^{\bar{u}} \wedge\tau^{u_\rho}\right)\right) \frac{\tau^{\bar{u}}-\tau^{u_\rho}}{\rho},
\end{aligned}
\end{equation}
where
\begin{equation}\label{Psi}
\begin{aligned}
\tilde{\Psi}^{\bar{u}}\left(t\right)=&\Psi_x\left(X^{\bar{u}}\left(t\right)\right)^{\top} f\left(X^{\bar{u}}\left(t\right), \bar{u}\left(t\right)\right)
+\frac{1}{2} \sum_{j=1}^d \sigma^j\left(X^{\bar{u}}\left(t\right), \bar{u}\left(t\right)\right)^{\top} \Psi_{x x}\left(X^{\bar{u}}\left(t\right)\right) \sigma^j\left(X^{\bar{u}}\left(t\right), \bar{u}\left(t\right)\right).
\end{aligned}
\end{equation}
(ii). If $\inf \{t: \mathbb{E}[\Phi(X^{\bar{u}}(t))] \geq \alpha, t \in[0, T]\}=T$, we set
that $\kappa^{\rho}\left(\tau^{\bar{u}} \wedge\tau^{u_\rho} \right)$ satisfies equation (\ref{condition1}) or
\begin{align*}
\begin{split}
\kappa^{\rho}\left(\tau^{\bar{u}}\wedge\tau^{u_{\rho}} \right)=&\Psi_x\left(X^{\bar{u}}\left(\tau^{\bar{u}}\wedge\tau^{u_{\rho}}\right)\right) \xi\left(\tau^{\bar{u}}\wedge\tau^{u_{\rho}}\right).
\end{split}
\end{align*}
(iii). If $\{t: \mathbb{E}[\Phi(X^{\bar{u}}(t))] \geq \alpha, t \in[0, T]\}=\varnothing$, we set
\begin{align}\label{nn2-3}
\begin{split}
\kappa^{\rho}\left(\tau^{\bar{u}}\wedge\tau^{u_{\rho}} \right)=&\Psi_x\left(X^{\bar{u}}\left(\tau^{\bar{u}}\wedge\tau^{u_{\rho}}\right)\right) \xi\left(\tau^{\bar{u}}\wedge\tau^{u_{\rho}}\right).
\end{split}
\end{align}

Furthermore, let $\left(\eta\left(\cdot \right), \zeta\left(\cdot \right) \right)$ satisfy the following equation
\begin{equation}\label{equation3}
\begin{cases}
\begin{aligned}
d \eta(t)=&\left(g_x(x(t), y(t), z(t), u(t)) \xi(t)+g_y(x(t), y(t), z(t), u(t)) \eta(t)\right.\\
&\left.+g_z(x(t), y(t), z(t), u(t)) \zeta(t)+g_v(x(t), y(t), z(t), u(t)) v(t)\right) dt \\
&+\zeta(t) d W(t),\\
\eta\left (\tau^{\bar{u}}\right)
= &\kappa(\tau^{\bar{u}}),
\end{aligned}
\end{cases}
\end{equation}
where
\begin{equation}\label{kappa}
\begin{aligned}
\kappa(\tau^{\bar{u}})= & -\int_0^{\tau^{\bar{u}}} \frac{\tilde{\Psi}^{\bar{u}}(\tau^{\bar{u}}) \bar{h}(v(s), s)}{h^{\bar{u}}\left(\tau^{\bar{u}}\right)} \mathrm{d} s+\Psi_x\left(X^{\bar{u}}(\tau^{\bar{u}})\right)^{\top} \xi(\tau^{\bar{u}}) \\ & +g\left(X^{\bar{u}}(\tau^{\bar{u}}), Y^{\bar{u}}(\tau^{\bar{u}}), Z^{\bar{u}}(\tau^{\bar{u}}), \bar{u}(\tau^{\bar{u}})\right) \int_0^{\tau^{\bar{u}}} \frac{\bar{h}(v(s), s)}{h^{\bar{u}}\left(\tau^{\bar{u}}\right)} \mathrm{d} s.
\end{aligned}
\end{equation}
If $\tau^{\bar{u}}$ satisfies the second possibility in the case $(ii)$ or the case $(iii)$, then equation (\ref{kappa}) reduces to the equation (\ref{nn2-3}).

From Theorem $2.1$ in Peng\cite{peng1993backward}, there exists a unique pair
$$
(\eta^{\rho}(\cdot), \zeta^{\rho}(\cdot)) \in \mathcal{M}^2 \left(\mathbb{R}^m\right) \times \mathcal{M}^2 \left(\mathbb{R}^{m \times d}\right)
$$
which solves equations (\ref{equation1}) and (\ref{equation2}), and there exists a unique pair
$$
(\eta(\cdot), \zeta(\cdot)) \in  \mathcal{M}^2 \left(\mathbb{R}^m\right) \times \mathcal{M}^2 \left(\mathbb{R}^{m \times d}\right)
$$
satisfying equations (\ref{equation1}) and (\ref{equation3}).

Now, we prove that $\left(\eta^{\rho}\left(\cdot \right), \zeta^{\rho}\left(\cdot \right) \right)$ converges to $\left(\eta\left(\cdot \right), \zeta\left(\cdot \right) \right)$ for $t \in [0, \tau^{u_\rho}\wedge\tau^{\bar{u}}]$ when $\rho \to 0$, which is useful in the proof of Theorem \ref{Theorem}.
\begin{lemma}\label{vari-limit}
Let Assumptions \ref{assumption_1}, \ref{assumption_2} and \ref{assumption_3} hold. $\tau^{\bar{u}}$ and $\tau^{u_{\rho}}$ is defined in equation (\ref{varying_time}). Then we have
\begin{align*}
&\lim_{\rho \to 0} \mathbb{E}\left [ \left |\eta^{\rho}\left(t\right)-\eta\left(t\right) \right |^2  \right ] =0,\\
&\lim_{\rho \to 0} \mathbb{E}\left [\int_{0}^{\tau^{\bar{u}}\wedge\tau^{u_\rho}}
 \left |\zeta^{\rho}\left(t\right)-\zeta\left(t\right) \right |^2 dt \right ] =0
\end{align*}
for $t \in [0, \tau^{u_\rho} \wedge \tau^{\bar{u}}]$.
\end{lemma}
\begin{proof}
Let $\tilde{\eta}\left(t\right)=\eta^{\rho}\left(t\right)-\eta\left(t\right)$, $\tilde{\zeta}\left(t\right)=\zeta^{\rho}\left(t\right)-\zeta\left(t\right)$ and
$\hat{\kappa}=\kappa^{\rho}(\tau^{\bar{u}}\wedge\tau^{u_\rho})-\kappa(\tau^{\bar{u}})$. By a simple calculation, we have
\begin{equation}\label{rho_equation}
\begin{cases}
\begin{aligned}
&d\tilde{\eta}\left(t\right) = \left(g_y(x(t), y(t), z(t), u(t))\tilde{\eta}(t)+g_z(x(t), y(t), z(t), u(t))\tilde{\zeta}(t)\right) dt+\tilde{\zeta}(t)dW(t), \\
&\tilde{\eta}\left(\tau^{\bar{u}}\wedge\tau^{u_\rho}\right)=\hat{\kappa}+\int_{\tau^{u_\rho} \wedge \tau^{\bar{u}}}^{\tau^{\bar{u}}}\left(g_x \xi(t)+g_y \eta(t)+g_z \zeta(t)+g_v v(t)\right) d t+\int_{\tau^{u_\rho} \wedge \tau^{\bar{u}}}^{\tau^{\bar{u}}} \zeta(t) d W(t).
\end{aligned}
\end{cases}
\end{equation}
For $\left|\tilde{\eta}(\tau^{u_\rho}\wedge \tau^{\bar{u}})\right|^2$, by Burkholder-Davis-Gundy inequality, it follows that
\begin{equation*}
\begin{aligned}
& \mathbb{E}\left|\tilde{\eta}\left(\tau^{u_\rho} \wedge \tau^{\bar{u}}\right)\right|^2 \\
= & \mathbb{E}\left|\hat{\kappa}+\int_{\tau^{u_\rho} \wedge \tau^{\bar{u}}}^{\tau^{\bar{u}}}\left(g_x \xi(t)+g_y \eta(t)+g_z \zeta(t)+g_v v(t)\right) d t+\int_{\tau^{u_\rho} \wedge \tau^{\bar{u}}}^{\tau^{\bar{u}}} \zeta(t) d W(t)\right|^2 \\
\leq & K\mathbb{E}|\hat{\kappa}|^2+K \mathbb{E} \int_{\tau^{u_\rho}\wedge \tau^{\bar{u}}}^{\tau^{\bar{u}}}\left[\left|g_x \xi(t)\right|^2+\left|g_y \eta(t)\right|^2+\left|g_z \zeta(t)\right|^2+\left|g_v v(t)\right|^2\right] d t+K \mathbb{E} \int_{\tau^{u_\rho} \wedge \tau^{\bar{u}}}^{\tau^{\bar{u}}}|\zeta(t)|^2 d t
\end{aligned}
\end{equation*}
where $K$ is a positive constant. Based on Assumptions \ref{assumption_1}, \ref{assumption_2} and \ref{assumption_3}, by Lemma \ref{lemma 3.2}, Lemma \ref{lemma 3.4} and dominated convergence theorem, we can obtain that $\lim_{\rho\to0}\mathbb{E}|\hat{\kappa}|^2=0$. From Lemma \ref{lemma 3.2}, we have $\lim_{\rho\to 0}\left|\tau^{u_\rho}-\tau^{\bar{u}}\right|=0$. Thus, by dominated convergence theorem, when $\rho \to 0$, we can obtain that
$$
\mathbb{E} \left|\tilde{\eta}(\tau^{u_\rho}\wedge \tau^{\bar{u}})\right|^2 \to 0.
$$
Based on the continuous dependence of BSDEs on terminal value $\tilde{\eta}(\tau^{u_\rho}\wedge \tau^{\bar{u}})$, it follows that
\begin{equation*}
\begin{aligned}
&\mathbb{E} \left|\tilde{\eta}(t)\right|^2 \to 0, \quad \mathbb{E} \int_t^{\tau^{u_\rho}\wedge \tau^{\bar{u}}}\left|\tilde{\zeta}(s)\right|^2 ds \to 0
\end{aligned}
\end{equation*}
when $\rho \to 0$ and t $\in [0, \tau^{u_\rho} \wedge \tau^{\bar{u}}]$.

This completes the proof.
\end{proof}

\begin{remark}
The terminal time varies according to $\rho$ and $\tau^{u_{\rho}}$ should be smaller or larger than $\tau^{\bar{u}}$ which makes it difficult to determine the terminal time of the related variational equations directly. Therefore, we propose a novel $\rho$-moving variational equations and prove that its solutions can converge to the solutions of equation (\ref{equation3}) for $t\in [0, \tau^{u_\rho} \wedge \tau^{\bar{u}}]$ as $\rho \to 0$.
\end{remark}
Let
$\bar{u}\left(\cdot \right) $ be an optimal control and $\left(X^{\bar{u}}\left(\cdot \right), Y^{\bar{u}}\left(\cdot \right), Z^{\bar{u}}\left(\cdot \right) \right)$ be the corresponding trajectory. Similarly, we denote by $\left(X^{u_\rho}(\cdot), Y^{u_\rho}(\cdot), Z^{u_\rho}(\cdot)\right)$ the trajectory corresponding to $u_\rho$. In the following, we set
$$
\begin{aligned}
\tilde{X}_\rho(t) \equiv \rho^{-1}\left(X^{u_\rho}(t)-X^{\bar{u}}(t)\right)-\xi(t)
\end{aligned}
$$
where $t \in \left [ 0,\tau^{\bar{u}}\right]$, and
$$
\begin{aligned}
&\tilde{Y}_\rho(t) \equiv \rho^{-1}\left(Y^{u_\rho}(t)-Y^{\bar{u}}(t)\right)-\eta^{\rho}(t), \\
&\tilde{Z}_\rho(t) \equiv \rho^{-1}\left(Z^{u_\rho}(t)-Z^{\bar{u}}(t)\right)-\zeta^{\rho}(t),
\end{aligned}
$$
where $t \in \left [ 0,\tau^{u_\rho} \wedge\tau^{\bar{u}} \right]$. We have the following convergence results.
\begin{lemma}\label{equation_lemma}
Let Assumptions \ref{assumption_1}, \ref{assumption_2} and \ref{assumption_3} hold. $\tau^{\bar{u}}$ is defined in equation (\ref{varying_time}), let $h^{\bar{u} }\left ( \tau^{\bar{u} } \right ) \ne 0$ and $h^{\bar{u} }\left ( \cdot  \right )$ be continuous at point $\tau^{\bar{u}}$. Therefore, we have
\begin{align*}
&\lim _{\rho \rightarrow 0} \sup _{0 \leq t \leq \tau^{\bar{u}}} \mathbb{E} \left|\tilde{X}_\rho(t)\right|^2=0, \\
&\lim _{\rho \rightarrow 0} \sup _{0 \leq t \leq \tau^{u_\rho} \wedge \tau^{\bar{u}}} \mathbb{E} \left|\tilde{Y}_\rho(t)\right|^2=0, \\
&\lim _{\rho \rightarrow 0} \mathbb{E}  \int_0^{\tau^{u_\rho} \wedge\tau^{\bar{u}}}\left|\tilde{Z}_\rho(t)\right|^2=0.
\end{align*}
\end{lemma}

\begin{proof}
Based on Lemma 4.1 in Peng\cite{peng1993backward}, we can obtain the convergence of $\tilde{X}_\rho(\cdot)$. In the following, we prove the convergence of $\tilde{Y}_\rho(\cdot)$ and $\tilde{Z}_\rho(\cdot)$. Let
\begin{align*}
&Y^{u_\rho} (t)=\Psi\left(X^{u_\rho}\left(\tau^{u_\rho}\right)\right)-\int_t^{\tau^{u_\rho}} g\left(X^{u_\rho}(s), Y^{u_\rho}(s), Z^{u_\rho}(s), u_{\rho}(s)\right) ds-\int_t^{\tau^{u_\rho}} Z^{u_\rho}(s) dW(s), \\
&Y^{\bar{u}} (t)=\Psi\left(X^{\bar{u}}\left(\tau^{\bar{u}}\right)\right)-\int_t^{\tau^{\bar{u}}} g\left(X^{\bar{u}}(s), Y^{\bar{u}}(s), Z^{\bar{u}}(s), \bar{u}(s)\right) ds-\int_t^{\tau^{\bar{u}}} Z^{\bar{u}}(s) dW(s).
\end{align*}
Note that
\begin{equation*}
\begin{cases}
\begin{aligned}
d\tilde{Y}_{\rho}\left(t\right)=&\rho^{-1}\left[g\left(X^{\bar{u}}+\rho\left(\xi+\tilde{X}_\rho\right),
Y^{\bar{u}}+\rho\left(\eta^{\rho}+\tilde{Y}_\rho\right), Z^{\bar{u}}+\rho\left(\zeta^{\rho}+\tilde{Z}_\rho\right),\bar{u}+\rho v, t\right)\right]  \\
&\left.-g(X^{\bar{u}}, Y^{\bar{u}}, Z^{\bar{u}}, \bar{u}, t)-g_x \xi-g_y \eta^{\rho}-g_z \zeta^{\rho}-g_v v\right] d t+\tilde{Z}(t)dW(t)\\
\tilde{Y}_{\rho}\left(\tau^{u_\rho}\wedge\right. &  \left.\tau^{\bar{u}}\right)= \rho^{-1}\left(Y^{u_\rho}(\tau^{u_\rho}\wedge \tau^{\bar{u}})-Y^{\bar{u}}(\tau^{u_\rho}\wedge \tau^{\bar{u}})\right)-\eta^{\rho}(\tau^{u_\rho}\wedge \tau^{\bar{u}})
\end{aligned}
\end{cases}
\end{equation*}
can be rewritten as
\begin{equation*}
\left\{
\begin{aligned}
d\tilde{Y}_{\rho}\left(t\right)=&
\left (A_{\rho}\left(t\right)\tilde{X}_{\rho}\left(t\right)
+ B_{\rho}\left(t\right)\tilde{Y}_{\rho}\left(t\right)+
C_{\rho}\left(t\right)\tilde{Z}_{\rho}\left(t\right)
+D_\rho\left ( t \right )  \right ) dt+
\tilde{Z}_{\rho}\left(t\right)dW(t),\\
\tilde{Y}_{\rho}\left(\tau^{u_\rho}\wedge\right. & \left.\tau^{\bar{u}}\right)= \rho^{-1}\left(Y^{u_\rho}(\tau^{u_\rho}\wedge \tau^{\bar{u}})-Y^{\bar{u}}(\tau^{u_\rho}\wedge \tau^{\bar{u}})\right)-\kappa^{\rho}(\tau^{\bar{u}}\wedge\tau^{u_\rho})
\end{aligned}
\right.
\end{equation*}
where
\begin{align*}
A_\rho=&\int_0^1\left[g_x(X^{\bar{u}}+\lambda \rho\left(\xi+\tilde{X}_\rho\right),
Y^{\bar{u}}+\lambda \rho\left(\eta^\rho+\tilde{Y}_\rho\right),
Z^{\bar{u}}+\lambda \rho\left(\zeta^\rho+\tilde{Z}_\rho\right),
\bar{u}+\lambda \rho v, t)\right] d \lambda,\\
B_\rho=&\int_0^1\left[g_y(X^{\bar{u}}+\lambda \rho\left(\xi+\tilde{X}_\rho\right),
Y^{\bar{u}}+\lambda \rho\left(\eta^\rho+\tilde{Y}_\rho\right),
Z^{\bar{u}}+\lambda \rho\left(\zeta^\rho+\tilde{Z}_\rho\right),
\bar{u}+\lambda \rho v, t)\right] d \lambda,\\
C_\rho=&\int_0^1\left[g_z(X^{\bar{u}}+\lambda \rho\left(\xi+\tilde{X}_\rho\right),
Y^{\bar{u}}+\lambda \rho\left(\eta^\rho+\tilde{Y}_\rho\right),
Z^{\bar{u}}+\lambda \rho\left(\zeta^\rho+\tilde{Z}_\rho\right),
\bar{u}+\lambda \rho v, t)\right] d \lambda,\\
D_\rho=&\left(A_\rho(t)-g_x\left(X^{\bar{u}}, Y^{\bar{u}}, Z^{\bar{u}}, \bar{u}, t\right)\right)\xi
+\left(B_\rho(t)-g_y\left(X^{\bar{u}}, Y^{\bar{u}}, Z^{\bar{u}}, \bar{u}, t\right)\right)\eta^\rho\\
&+\left(C_\rho(t)-g_z\left(X^{\bar{u}}, Y^{\bar{u}}, Z^{\bar{u}}, \bar{u}, t\right)\right)\zeta^\rho+ \int_0^1\left[g_v(X^{\bar{u}}+\lambda \rho\left(\xi+\tilde{X}_\rho\right),
Y^{\bar{u}}+\lambda \rho\left(\eta^\rho+\tilde{Y}_\rho\right),\right.\\
&\left.Z^{\bar{u}}+\lambda \rho\left(\zeta^\rho+\tilde{Z}_\rho\right),
\bar{u}+\lambda \rho v, t)-g_v\left(X^{\bar{u}}, Y^{\bar{u}}, Z^{\bar{u}}, \bar{u}, t\right) \right] v d\lambda.
\end{align*}

In the following, we consider three kinds of case of $\tau^{\bar{u}}$ which is given in Lemma \ref{lemma 3.4}. \textbf{Firstly, we prove the case $(i)$}: $\inf \left\{t: \mathbb{E}[\Phi(X^{\bar{u}}(t))] \geq \alpha, t \in[0, T]\right\}<T$. Applying It\^{o}'s formula to $\left | \tilde{Y}_\rho\left ( \cdot  \right )  \right | ^2$, it follows that
\begin{align}
&\mathbb{E} \left|\tilde{Y}_\rho(t)\right|^2
+\mathbb{E} \int_t^{\tau^{u_\rho}\wedge \tau^{\bar{u}}}\left|\tilde{Z}_\rho(s)\right|^2 ds \label{grownwall1}\\
=&\mathbb{E} \left|\tilde{Y}_\rho(\tau^{u_\rho}\wedge \tau^{\bar{u}})\right|^2
-2 E \int_t^{\tau^{u_\rho}\wedge \tau^{\bar{u}}}\left(\tilde{Y}_\rho(s), A_\rho(s) \tilde{X}_\rho(s)+B_\rho(s) \tilde{Y}_\rho(s)+C_\rho(s) \tilde{Z}_\rho(s)+D_\rho(s)\right) ds \nonumber\\
\le& \mathbb{E} \left|\tilde{Y}_\rho(\tau^{u_\rho}\wedge \tau^{\bar{u}})\right|^2
+K \mathbb{E} \int_t^{\tau^{u_\rho}\wedge \tau^{\bar{u}}}\left|\tilde{Y}_\rho(s)\right|^2 d s+2^{-1} E \int_t^{\tau^{u_\rho}\wedge \tau^{\bar{u}}}\left|\tilde{Z}_\rho(s)\right|^2 d s+J_\rho \nonumber
\end{align}
where
$$
J_\rho=\mathbb{E} \int_t^{\tau^{u_\rho}\wedge \tau^{\bar{u}}}\left(\left|A_\rho(s) \tilde{X}_\rho(s)\right|^2+\left|D_\rho(s)\right|^2\right) ds.
$$

To distinguish the term $\tau^{\bar{u}}$ and $\tau^{u_\rho}$, we consider two different situations. Suppose $\tau^{\bar{u}} < \tau^{u_\rho}$. For the item $\mathbb{E}\left|\tilde{Y}_\rho(\tau^{\bar{u}} \wedge \tau^{u_\rho} )\right|^2$, given that $\tilde{Y}_\rho(\tau^{\bar{u}} \wedge \tau^{u_\rho})$ is $\mathcal{F}_{\tau^{\bar{u}} \wedge \tau^{u_\rho}}$-measurable, we have
\begin{align}
&\mathbb{E}\left|\tilde{Y}_\rho(\tau^{\bar{u}} \wedge \tau^{u_\rho})\right|^2=
\mathbb{E}\left|\mathbb{E}[\tilde{Y}_\rho(\tau^{\bar{u}} \wedge \tau^{u_\rho})\mid\mathcal{F}_{\tau^{\bar{u}}\wedge\tau^{u_\rho}}]\right|^2 \label{var1}\\
=&\mathbb{E}\left[\mathbb{E}\left[\frac{1}{\rho}\left ( Y^{u_\rho}\left ( \tau^{\bar{u}} \right ) -Y^{\bar{u}}\left ( \tau^{\bar{u}} \right )  \right )
-\kappa^{\rho}(\tau^{\bar{u}}\wedge\tau^{u_\rho})\mid\mathcal{F}_{\tau^{\bar{u}}}\right]
\right]^2 \nonumber \\
=&\mathbb{E}\left[ \mathbb{E}\left[
\frac{1}{\rho}\left(\Psi(X^{u_\rho}\left ( \tau^{u_\rho} \right )) -\Psi(X^{\bar{u}}\left ( \tau^{\bar{u}} \right ))  \right )
-\frac{1}{\rho}\int_{\tau^{\bar{u}}}^{\tau^{u_\rho}}
g\left (X^{u_\rho}\left (s \right ),Y^{u_\rho}\left (s \right ),Z^{u_\rho}\left (s \right ),u_\rho\left (s \right )\right ) ds\right.\right.\nonumber \\
&\left.\left.-\frac{1}{\rho}\int_{\tau^{\bar{u}}}^{\tau^{u_\rho}}Z^{u_\rho}\left (s \right )dW(s)-\kappa^{\rho}(\tau^{\bar{u}}\wedge\tau^{u_\rho})\mid\mathcal{F}_{\tau{\bar{u}}}\right]
\right]^2\nonumber \\
=&\mathbb{E}\left[\mathbb{E}\left[
\frac{1}{\rho}\left(\Psi(X^{u_\rho}\left ( \tau^{u_\rho} \right )) -\Psi(X^{u_\rho}\left ( \tau^{\bar{u}} \right ))  \right ) +\frac{1}{\rho}\left(\Psi(X^{u_\rho}\left ( \tau^{\bar{u}} \right )) -\Psi(X^{\bar{u}}\left ( \tau^{\bar{u}} \right ))  \right )\right.\right. \nonumber \\
&\left.\left.-\frac{1}{\rho}\int_{\tau^{\bar{u}}}^{\tau^{u_\rho}}
g\left (X^{u_\rho}\left (s \right ),Y^{u_\rho}\left (s \right ),Z^{u_\rho}\left (s \right ),u_\rho\left (s \right )\right ) ds-\kappa^{\rho}(\tau^{\bar{u}}\wedge\tau^{u_\rho})\mid\mathcal{F}_{\tau^{\bar{u}}}\right]
\right]^2.\nonumber
\end{align}
Applying $It\hat{o}$'s formula to $\Psi\left(X^{u_\rho}(t)\right)$, we can obtain that
\begin{align}
\Psi(X^{u_\rho}\left ( \tau^{u_\rho} \right )) -\Psi(X^{u_\rho}\left ( \tau^{\bar{u}} \right ))=&
\int_{\tau^{\bar{u}}}^{\tau^{u_\rho}}\Big[
\Psi_x(X^{u_\rho}(t))^\top f(X^{u_\rho}(t), u_\rho(t))\Big.\label{ito1}\\
&\Big.+\frac{1}{2}
\sum_{j=1}^{d}\sigma^j(X^{u_\rho}(t), u_\rho(t))^\top \Psi_{xx}(X^{u_\rho}(t))\sigma^j(X^{u_\rho}(t), u_\rho(t))\Big]dt\nonumber\\
&+\int_{\tau^{\bar{u}}}^{\tau^{u_\rho}}\Psi_x(X^{u_\rho}(t))^\top \sigma(X^{u_\rho}(t), u_\rho(t))dW(t),\nonumber
\end{align}
and we denote that
$$
\tilde{\Psi}^{u_\rho}(t)=\Psi_x(X^{u_\rho}(t))^\top f(X^{u_\rho}(t), u_\rho(t))
+\frac{1}{2}
\sum_{j=1}^{d}\sigma^j(X^{u_\rho}(t), u_\rho(t))^\top \Psi_{xx}(X^{u_\rho}(t))\sigma^j(X^{u_\rho}(t), u_\rho(t)).
$$
Applying equation (\ref{ito1}) and Jensen's inequality to formula (\ref{var1}), we have
\begin{equation*}
\begin{aligned}
\mathbb{E}\left|\tilde{Y}_\rho(\tau^{\bar{u}} \wedge \tau^{u_\rho})\right|^2\le
\mathbb{E}\Big[\Big(&\frac{1}{\rho}\int_{\tau^{\bar{u}}}^{\tau^{u_\rho}}\tilde{\Psi}^{u_\rho}(t)dt
+\frac{1}{\rho}\left(\Psi(X^{u_\rho}\left ( \tau^{\bar{u}} \right )) -\Psi(X^{\bar{u}}\left ( \tau^{\bar{u}} \right ))  \right )\Big.\Big.\\
&\Big.\Big.-\frac{1}{\rho}\int_{\tau^{\bar{u}}}^{\tau^{u_\rho}}
g\left (X^{u_\rho}\left (s \right ),Y^{u_\rho}\left (s \right ),Z^{u_\rho}\left (s \right ),u_\rho\left (s \right )\right ) ds-\kappa^{\rho}(\tau^{\bar{u}}\wedge\tau^{u_\rho})\Big)^2\Big].
\end{aligned}
\end{equation*}
By Lemma 3 in Yang\cite{yang2020varying}, Lemma 4.1 in Peng \cite{peng1993backward} and Lemma \ref{lemma 3.4}, we have
\begin{equation}\label{n1}
 \mathbb{E}\left[\frac{1}{\rho}\int_{\tau^{\bar{u}}}^{\tau^{u_\rho}}\tilde{\Psi}^{u_\rho}(t)dt
 -\Big(
 \frac{\tilde{\Psi}^{\bar{u}}\left(\tau^{\bar{u}}\wedge\tau^{u_\rho}\right)(\tau^{u_\rho}-\tau^{\bar{u}})}{\rho}\Big)
\right]^2 \to 0  \ (\rho \to 0).
\end{equation}
Using Lemma 5 in Yang\cite{yang2020varying} and Lemma 4.1 in Peng \cite{peng1993backward}, we have
\begin{equation}\label{n2}
\mathbb{E}\left[\left(\frac{\Psi\left(X^{u_\rho}\left(\tau^{\bar{u}}\right)\right)-\Psi\left(X^{\bar{u}}\left(\tau^{\bar{u}}\right)\right)}
{\rho}\right)-\Psi_x\left(X^{\bar{u}}\left(\tau^{\bar{u}}\wedge\tau^{u_\rho}\right)\right) \xi\left(\tau^{\bar{u}}\wedge\tau^{u_\rho}\right)
\right]^2  \to 0  \ (\rho \to 0).
\end{equation}
Similarly, using Lemma \ref{3.4} and the continuous dependence of BSDEs on terminal value, we can obtain that
 \begin{align}
&\frac{1}{\rho}\int_{\tau^{\bar{u}}}^{\tau^{u_\rho}}
g\left (X^{u_\rho}\left (s \right ),Y^{u_\rho}\left (s \right ),Z^{u_\rho}\left (s \right ),u_\rho\left (s \right )\right ) ds \nonumber\\
 &\overset{L^2}{\longrightarrow} g\left(X^{\bar{u}}\left(\tau^{\bar{u}}\wedge\tau^{u_\rho}\right), Y^{\bar{u}}\left(\tau^{\bar{u}}\wedge\tau^{u_\rho}\right), Z^{\bar{u}}\left(\tau^{\bar{u}}\wedge\tau^{u_\rho}\right), \bar{u}\left(\tau^{\bar{u}}\wedge\tau^{u_\rho}\right)\right)\frac{\tau^{u_\rho}-\tau^{\bar{u}}}{\rho} \ \ (\rho \to 0). \label{n4}
\end{align}
Note that $\kappa^{\rho}(\tau^{\bar{u}}\wedge\tau^{u_\rho})$ satisfies (\ref{condition1}), we have
$$
\mathbb{E}\left|\tilde{Y}_\rho(\tau^{\bar{u}} \wedge \tau^{u_\rho} )\right|^2 \to 0 \ \ (\rho \to 0).
$$
For the term $J_\rho$, it is easy to show that $\lim_{\rho \to 0}J_\rho = 0$. Applying Gronwall's inequality to equation (\ref{grownwall1}), it follows that
\begin{equation}
\label{con1}
 \begin{aligned}
&\lim_{\rho \to 0} \mathbb{E}\left | \tilde{Y}_\rho\left ( t \right )   \right |^2
=0\\
&\lim_{\rho \to 0} \mathbb{E}\int_{t}^{\tau^{\bar{u}} \wedge \tau^{u_\rho}} \left | \tilde{Z}_\rho\left ( s \right )   \right |^2 ds
=0
\end{aligned}
\end{equation}
for $t \in [0, \tau^{\bar{u}} \wedge \tau^{u_\rho}]$ where $\tau^{\bar{u}}<\tau^{u_\rho}$.

Suppose $\tau^{u_\rho} < \tau^{\bar{u}}$. Using the similar manner in the proof of case $\tau^{\bar{u}}<\tau^{u_\rho}$, we can obtain equation (\ref{con1}).

\textbf{Secondly, we consider the case $(ii)$}. Notice that $\inf \left\{t: \mathbb{E}[\Phi(X^{\bar{u}}(t))] \geq \alpha, t \in[0, T]\right\}=T$. If there exists sequence $\rho_n \to 0$ as $n \to +\infty$ such that
$$
\tau^{u_{\rho_n}}=\inf \left\{t: \mathbb{E}\left[\Phi\left(X^{u_{\rho_n}}(t)\right)\right] \geq \alpha, t \in[0, T]\right\}<T,
$$
we can obtain the convergence results similar to the case $(i)$ and $\kappa^{\rho}(\tau^{\bar{u}} \wedge \tau^{u_\rho})$ satisfies equation (\ref{condition1}).

If there exists sequence $\rho_n \rightarrow 0$ as $n \rightarrow+\infty$ such that
$$
\inf \left\{t: \mathbb{E}\left[\Phi\left(X^{u_{\rho_n}}(t)\right)\right] \geq \alpha, t \in[0, T]\right\}=+\infty,
$$
we can obtain $\tau^{u_{\rho_n}}=\tau^{\bar{u}}=T$. There exists
\begin{align}\label{3.26}
\begin{split}
&\mathbb{E} \left|\tilde{Y}_\rho(t)\right|^2
+\mathbb{E} \int_t^T\left|\tilde{Z}_\rho(s)\right|^2 ds\\
=&\mathbb{E} \left|\tilde{Y}_\rho(T)\right|^2
-2 E\int_t^T\left(\tilde{Y}_\rho(s), A_\rho(s) \tilde{X}_\rho(s)+B_\rho(s) \tilde{Y}_\rho(s)+C_\rho(s) \tilde{Z}_\rho(s)+D_\rho(s)\right) ds\\
\le& \mathbb{E} \left|\tilde{Y}_\rho(T)\right|^2
+K E \int_t^{T}\left|\tilde{Y}_\rho(s)\right|^2 d s+2^{-1} E \int_t^T\left|\tilde{Z}_\rho(s)\right|^2 d s+J_\rho
\end{split}
\end{align}
where
$$
J_\rho=\mathbb{E} \int_t^T\left(\left|A_\rho(s) \tilde{x}_\rho(s)\right|^2+\left|D_\rho(s)\right|^2\right) ds.
$$
For the item $\mathbb{E}\left|\tilde{Y}_\rho(T)\right|^2$, combing equation (\ref{2.4}), one can obtain
\begin{align}
\mathbb{E}\left|\tilde{Y}_\rho(T)\right|^2=&
\mathbb{E}\left[\Big(\frac{1}{\rho}\left ( Y^{u_\rho}\left ( \tau^{\bar{u}} \right ) -Y^{\bar{u}}\left ( \tau^{\bar{u}} \right )  \right )
-\kappa^{\rho}(\tau^{\bar{u}} \wedge \tau^{u_\rho})
\Big)^2\right] \nonumber\\
=&\mathbb{E}\left[\Big(
\frac{1}{\rho}\left(\Psi(X^{u_\rho}\left ( T \right ))
-\Psi(X^{\bar{u}}\left ( T\right ))\right )-\kappa^{\rho}(\tau^{\bar{u}} \wedge \tau^{u_\rho})\Big)^2\right] \label{nn1}.
\end{align}
Applying Lemma 5 in Yang\cite{yang2020varying} to equation (\ref{nn1}), and note that $\kappa^{\rho}(\tau^{\bar{u}} \wedge \tau^{u_\rho})=\Psi_x\left(X^{\bar{u}}\left(\tau^{\bar{u}}\wedge\tau^{u_{\rho}}\right)\right) \notag \\ \xi\left(\tau^{\bar{u}}\wedge\tau^{u_{\rho}}\right)$. Thus, we can obtain that $\mathbb{E}\left|\tilde{Y}_\rho(T)\right|^2 \to 0$ when $\rho \to 0$.
By Gronwall's inequality, we can prove the convergence of the last two term in equation (\ref{3.26}).

\textbf{Finally, we consider the case $(iii)$}, where $\{t:\mathbb{E}[\Phi(X^{\bar{u}}(t))] \geq \alpha, t \in[0, T]\}=\varnothing$. For sufficiently small $\rho$,
$$
\inf \left\{t: \mathbb{E}\left[\Phi\left(X^{u_\rho}(t)\right)\right] \geq \alpha, t \in[0, T]\right\}=+\infty.
$$
We have $\tau^{u_\rho}=\tau^{\bar{u}}=T$, which is similar to the second condition in case $(ii)$ and we omit the proof.

This completes the proof.
\end{proof}
We now derive the variational equation for cost functional (\ref{cost_functional}). Note that $\bar{u}(\cdot)$ is the optimal control defined in equation (\ref{min_cost_functional}), we have
\begin{equation}\label{variational_equ}
\rho^{-1}\left [ J\left ( \bar{u}\left ( \cdot  \right ) +
\rho v\left ( \cdot  \right )  \right) -J\left ( \bar{u}\left ( \cdot  \right )  \right ) \right ]
\ge 0.
\end{equation}

\begin{lemma}\label{lemma 3.7}
Let Assumptions \ref{assumption_1}, \ref{assumption_2} and \ref{assumption_3} hold. $\tau^{\bar{u}}$ is defined in equation (\ref{varying_time}), and let $h^{\bar{u} }\left ( \tau^{\bar{u} } \right ) \ne 0$ and $h^{\bar{u} }\left ( \cdot  \right )$ be continuous at point $\tau^{\bar{u}}$. We have the following results.

(i). If $\tau^{\bar{u}} < T$, when $\rho \to 0$, one can obtain
\begin{align}\label{ineq-1}
\begin{split}
\mathbb{E}&\Big[\int_0^{\tau^{\bar{u}}}\Big(l_x\left(X^{\bar{u}}(t), Y^{\bar{u}}(t), Z^{\bar{u}}(t), \bar{u} (t), t\right)\xi(t)+l_y\left(X^{\bar{u}}(t), Y^{\bar{u}}(t), Z^{\bar{u}}(t), \bar{u}(t), t\right)\eta(t)\Big.\Big.\\
&\Big.\Big.+l_z\left(X^{\bar{u}}(t), Y^{\bar{u}}(t), Z^{\bar{u}}(t), \bar{u}(t), t\right)\zeta(t)+l_v\left(X^{\bar{u}}(t), Y^{\bar{u}}(t), Z^{\bar{u}}(t), \bar{u}(t), t\right)v(t)\Big.\Big.\\
&\Big.\Big.-\frac{\bar{h}(v(t), t) \tilde{\beta}^{\bar{u}}\left(\tau^{\bar{u}}\right)}{h^{\bar{u}}\left(\tau^{\bar{u}}\right)}
-\mathbb{E}[l\left(X^{\bar{u}}(\tau^{\bar{u}}), Y^{\bar{u}}(\tau^{\bar{u}}), Z^{\bar{u}}(\tau^{\bar{u}}), \bar{u}(\tau^{\bar{u}}), \tau^{\bar{u}}\right)] \frac{\bar{h}(v(t), t)}{h^{\bar{u}}\left(\tau^{\bar{u}}\right)}\Big) dt\Big]\\
&+\mathbb{E}\left[\beta_x\left(X^{\bar{u}}\left(\tau^{\bar{u}}\right)\right)^{\top} \xi\left(\tau^{\bar{u}}\right)\right]+\mathbb{E}\left[\gamma_y(Y^{\bar{u}}\left(0\right))^{\top} \eta(0)\right]\ge 0.
\end{split}
\end{align}
where $\tilde{\beta}^{\bar{u}}$ satisfies
\begin{equation}\label{beta}
\tilde{\beta}^{\bar{u}}(t)=\mathbb{E}\Big[\beta_x\left(X^{\bar{u}}(t)\right)^{\top} f\left(X^{\bar{u}}(t), \bar{u}(t)\right)+\frac{1}{2} \sum_{j=1}^d \sigma^j\left(X^{\bar{u}}(t), \bar{u}(t)\right)^{\top} \beta_{xx}\left(X^{\bar{u}}(t)\right) \sigma^j\left(X^{\bar{u}}(t), \bar{u}(t)\right)\Big].
\end{equation}

(ii). If $inf\left\{t:\mathbb{E}[\Phi(X^{\bar{u}}(t))]\ge \alpha,t \in [0,T]\right\}=T$, when $\rho \to 0$, one can obtain inequality (\ref{ineq-1}) or (\ref{ineq-2})

(iii). If $\left\{t:\mathbb{E}[\Phi(X^{\bar{u}}(t))]\ge \alpha,t \in [0,T]\right\}=\emptyset$, when $\rho \to 0$, one can obtain
\begin{align}\label{ineq-2}
\begin{split}
&\mathbb{E}\Big[\int _ 0 ^ T \left(l_x\left(X^{\bar{u}}(t), Y^{\bar{u}}(t), Z^{\bar{u}}(t), \bar{u}(t), t\right) \xi(t)+l_y\left(X^{\bar{u}}(t), Y^{\bar{u}}(t), Z^{\bar{u}}(t), \bar{u}(t), t\right) \eta(t)\Big.\right.\\
&\left.\Big.+l_z\left(X^{\bar{u}}(t), Y^{\bar{u}}(t), Z^{\bar{u}}(t), \bar{u}(t), t\right) \zeta(t)+l_v\left(X^{\bar{u}}(t), Y^{\bar{u}}(t), Z^{\bar{u}}(t), \bar{u}(t), t\right) v(t)\right) d t \Big]\\
&+\mathbb{E}\left[\beta_x\left(X^{\bar{u}}(T)\right)\xi(T)\right]
+\mathbb{E}\left[\gamma_y\left(Y^{\bar{u}}(0)\right)\eta(0)\right]
\ge 0.
\end{split}
\end{align}
\end{lemma}

\begin{proof}
\textbf{Firstly, we consider the case $(i)$}. By equation (\ref{variational_equ}), we can obtain
\begin{align}\label{casei}
\begin{split}
&\rho^{-1}\left [ J\left ( \bar{u}\left ( \cdot  \right ) +
\rho v\left ( \cdot  \right )  \right) -J\left ( \bar{u}\left ( \cdot  \right )  \right ) \right ] \\
=& \rho^{-1}\mathbb{E}\left[\int_0^{\tau^{u_\rho}}l(X^{u_\rho}(t), Y^{u_\rho}(t), Z^{u_\rho}(t), u_\rho(t), t)dt
-\int_0^{\tau^{\bar{u}}}l(X^{\bar{u}}(t), Y^{\bar{u}}(t), Z^{\bar{u}}(t), \bar{u}(t), t)dt\right]\\
&+\rho^{-1}\mathbb{E}\left[\beta(X^{u_\rho}(\tau^{u_\rho}))-\beta(X^{\bar{u}}(\tau^{\bar{u}}))\right]+\rho^{-1}\mathbb{E}\left[\gamma(Y^{u_\rho}(0))-\gamma(Y^{\bar{u} }(0))
\right] \ge 0.
\end{split}
\end{align}
For notation simplicity, we denote
\begin{align*}
I_1 = &\rho^{-1}\mathbb{E}\left[\beta(X^{u_\rho}(\tau^{u_\rho}))-\beta(X^{\bar{u}}(\tau^{\bar{u} }))\right],\\
I_2 = &\rho^{-1}\mathbb{E}\left[\gamma(Y^{u_\rho}(0))-\gamma(Y^{\bar{u} }(0))\right],\\
I_3 = &\rho^{-1}\mathbb{E}\left[\int_0^{\tau^{u_\rho}}l(X^{u_\rho}(t), Y^{u_\rho}(t), Z^{u_\rho}(t), u_\rho(t), t)dt
-\int_0^{\tau^{\bar{u}}}l(X^{\bar{u}}(t), Y^{\bar{u}}(t), Z^{\bar{u}}(t), \bar{u}(t), t)dt\right].
\end{align*}

Firstly, we consider the term $I_1$ and $I_2$. Using the similar manner in the proof of Lemma \ref{equation_lemma}, for term $I_1$, we have
\begin{align*}
I_1 =&\rho^{-1}\mathbb{E}\left[\beta(X^{u_\rho}(\tau^{u_\rho}))-\beta (X^{\bar{u}}(\tau^{\bar{u} }))\right]\\
=&\rho^{-1}\mathbb{E}\left[\beta (X^{u_\rho}(\tau^{u_\rho}))-\beta (X^{u_\rho}(\tau^{\bar{u} }))\right]+\rho^{-1}\mathbb{E}\left[\beta(X^{u_\rho}(\tau^{\bar{u}}))-\beta(X^{\bar{u}}(\tau^{\bar{u} }))\right].
\end{align*}
Applying It\^{o}'s formula to $\beta(\cdot)$, we have
\begin{equation*}
\begin{cases}
\begin{aligned}
&d\beta\left(X^{u_\rho}(t)\right)=\beta _x\left(X^{u_\rho}(t))\right)^{\top}d X^{u_\rho}+\frac{1}{2}dX^{u_\rho}(t)^{\top} \beta _{x x}\left(X^{u_\rho}(t)\right) dX^{u_\rho}(t),\\
&dX^{u_\rho}(t)=b\left(X^{u_\rho}(t),u_\rho(t)\right)dt + \sigma\left(X^{u_\rho}(t), u_\rho(t)\right) d W_t.
\end{aligned}
\end{cases}
\end{equation*}
We denote
$$
\tilde{\beta}^{u_\rho}(t)=\mathbb{E}\left[\beta_x\left(X^{u_\rho}(t)\right)^{\top} f\left(X^{u_\rho}(t), u_\rho(t)\right)+\frac{1}{2} \sum_{j=1}^d \sigma^j\left(X^{u_\rho}(t), u_\rho(t)\right)^{\top} \beta_{xx}\left(X^{u_\rho}(t)\right) \sigma^j\left(X^{u_\rho}(t), u_\rho(t)\right)\right].
$$
Thus, one can obtain
\begin{equation*}
\rho^{-1}\mathbb{E}\left[\beta\left(X^{u_\rho}\left(\tau^{u_\rho}\right)\right)-\beta\left(X^{u_\rho}\left(\tau^{\bar{u}}\right)\right)\right]
= \rho^{-1}\int_{\tau^{\bar{u}}}^{\tau^{u_\rho}}
\tilde{\beta}^{u_\rho}(t)dt=\frac{\tau^{u_\rho}-\tau^{\bar{u}}}{\rho}(\tilde{\beta}^{u_\rho}(\tau^{\bar{u}})+o(1))
\end{equation*}
By Lemma \ref{lemma 3.4}, when $\rho \to 0$, we have
\begin{equation}
\label{n21}
\rho^{-1}\mathbb{E}\left[\beta\left(X^{u_\rho}\left(\tau^{u_\rho}\right)\right)-\beta\left(X^{u_\rho}\left(\tau^{\bar{u}}\right)\right)\right]
\to -\int_0^{\tau^{\bar{u}}} \tilde{\beta}^{\bar{u}}(\tau^{\bar{u}})\frac{\bar{h}(v(t), t)}{h^{\bar{u}}\left(\tau^{\bar{u}}\right)}dt.
\end{equation}
Using Lemma \ref{equation_lemma}, when $\rho \to 0$, it follows that
\begin{equation}
\label{n22}
\begin{aligned}
&\rho^{-1}\mathbb{E}\left[\beta (X^{u_\rho}(\tau^{\bar{u}}))-\beta (X^{\bar{u}}(\tau^{\bar{u} }))\right]\\
=&\rho^{-1}\mathbb{E}\left[(X^{u_\rho}(\tau^{\bar{u}})-X^{\bar{u}}(\tau^{\bar{u} }))(\beta_x(X^{\bar{u}}(\tau^{\bar{u} })) + o(1))\right] \to \mathbb{E}\left[\beta_x(X^{\bar{u}}(\tau^{\bar{u} }))^{\top}
\xi(\tau^{\bar{u} })\right].
\end{aligned}
\end{equation}
Combing equations (\ref{n21}) and (\ref{n22}), when $\rho \to 0$, it follows that
\begin{equation}
\label{re11}
\begin{aligned}
I_1 =& \rho^{-1}\mathbb{E}\left[\beta (X^{u_\rho}(\tau^{u_\rho}))-\beta (X^{\bar{u}}(\tau^{\bar{u}}))\right] \\
\to &-\int_0^{\tau^{\bar{u}}} \frac{\bar{h}(v(t), t)\tilde{\beta}(\tau^{\bar{u}})}{h^{\bar{u}}\left(\tau^{\bar{u}}\right)}dt + \mathbb{E}\left[\beta_x(X^{\bar{u}}(\tau^{\bar{u} }))^{\top}
\xi(\tau^{\bar{u} })\right].
\end{aligned}
\end{equation}

For the term $I_2$, using Lemma \ref{vari-limit} and \ref{equation_lemma}, when $\rho \to 0$, it follows that
\begin{equation}
\label{re21}
\begin{aligned}
I_2 = &\rho^{-1} \mathbb{E}\left[\gamma\left(Y^{u_\rho}(0)\right)-\gamma\left(Y^{\bar{u}}(0)\right)\right]\\
=&\rho^{-1}\mathbb{E}\left[\left(Y^{u_\rho}\left(0\right)-Y^{\bar{u}}\left(0\right)\right)\left(\gamma_y\left(Y^{\bar{u}}\left(0\right)\right)+o(1)\right)\right]\to \mathbb{E}\left[\gamma_y\left(Y^{\bar{u}}\left(0\right)\right)^{\top}\eta\left(0\right)\right].
\end{aligned}
\end{equation}
For the term $I_3$, we have
\begin{align*}
I_3=&\rho^{-1} \mathbb{E}[\int_0^{\tau^{u_\rho}} l\left(X^{u_\rho}(t), Y^{u_\rho}(t), Z^{u_\rho}(t), u_\rho(t), t\right) d t-\int_0^{\tau^{\bar{u}}} l\left(X^{\bar{u}}(t), Y^{\bar{u}}(t), z(t), \bar{u}(t), t\right) d t]\\
=&\rho^{-1} \mathbb{E}[\int_0^{\tau^{u_\rho}\wedge\tau^{\bar{u} }} (l\left(X^{u_\rho}(t), Y^{u_\rho}(t), Z^{u_\rho}(t), u_\rho(t), t\right)-l\left(X^{\bar{u}}(t), Y^{\bar{u}}(t), z(t), \bar{u}(t), t\right)) dt]\\
&+\rho^{-1} \mathbb{E}[\int_{\tau^{u_\rho}\wedge\tau^{\bar{u} }}^{\tau^{u_\rho}} l\left(X^{u_\rho}(t), Y^{u_\rho}(t), Z^{u_\rho}(t), u_\rho(t), t\right)dt]\\
&-\rho^{-1} \mathbb{E}[\int_{\tau^{u_\rho}\wedge\tau^{\bar{u} }}^{\tau^{\bar{u}}} l\left(X^{\bar{u}}(t), Y^{\bar{u}}(t), Z^{\bar{u}}(t), \bar{u}(t), t\right)dt].
\end{align*}

Consider the case that $\tau^{u_\rho}>\tau^{\bar{u}}$. Using Lemma \ref{lemma 3.4}, when $\rho \to 0$, it follows that
\begin{equation}\label{n31}
\begin{aligned}
&\rho^{-1} \mathbb{E}[\int_{\tau^{\bar{u}}}^{\tau^{u_\rho}} l\left(X^{u_\rho}(t), Y^{u_\rho}(t), Z^{u_\rho}(t), u_\rho(t), t\right)dt] \\
= &\rho^{-1} \mathbb{E}[(\tau^{u_\rho}-\tau^{\bar{u}}) (l\left(X^{u_\rho}(\tau^{\bar{u}}), Y^{u_\rho}(\tau^{\bar{u}}), Z^{u_\rho}(\tau^{\bar{u}}), u_\rho(\tau^{\bar{u}}), \tau^{\bar{u}}\right) + o(1))]\\
\to&-\int_0^{\tau^{\bar{u}}} \mathbb{E}[l\left(X^{\bar{u}}(\tau^{\bar{u}}), Y^{\bar{u}}(\tau^{\bar{u}}), Z^{\bar{u}}(\tau^{\bar{u}}), \bar{u}(\tau^{\bar{u}}), \tau^{\bar{u}}\right)] \frac{\bar{h}(v(t), t)}{h^{\bar{u}}\left(\tau^{\bar{u}}\right)}dt.
\end{aligned}
\end{equation}
Using Lemma \ref{vari-limit} and \ref{equation_lemma}, when $\rho \to 0$,
\begin{equation}\label{n32}
\begin{aligned}
&\rho^{-1} \mathbb{E}[\int_0^{\tau^{\bar{u}}\wedge\tau^{u_\rho}} (l\left(X^{u_\rho}(t), Y^{u_\rho}(t), Z^{u_\rho}(t), u_\rho(t), t\right)-l\left(X^{\bar{u}}(t), Y^{\bar{u}}(t), z(t), \bar{u}(t), t\right)) dt]\\
\to& \mathbb{E}[\int_0^{\tau^{\bar{u}}}\Big(l_x\left(X^{\bar{u}}(t), Y^{\bar{u}}(t), Z^{\bar{u}}(t), \bar{u}(t), t\right)\xi(t)+l_y\left(X^{\bar{u}}(t), Y^{\bar{u}}(t), Z^{\bar{u}}(t), \bar{u}(t), t\right)\eta(t)\Big.\\
&\Big.+l_z\left(X^{\bar{u}}(t), Y^{\bar{u}}(t), Z^{\bar{u}}(t), \bar{u}(t), t\right)\zeta(t)+l_v\left(X^{\bar{u}}(t), Y^{\bar{u}}(t), Z^{\bar{u}}(t), \bar{u}(t), t\right)v(t)\Big)dt].
\end{aligned}
\end{equation}
Combining equations (\ref{re11}), (\ref{re21}), (\ref{n31}) and (\ref{n32}), letting $\rho \to 0$, we can obtain
\begin{align}\label{conditionI}
\begin{split}
\mathbb{E}&\Big[\int_0^{\tau^{\bar{u}}}\Big(l_x\left(X^{\bar{u}}(t), Y^{\bar{u}}(t), Z^{\bar{u}}(t), \bar{u} (t), t\right)\xi(t)+l_y\left(X^{\bar{u}}(t), Y^{\bar{u}}(t), Z^{\bar{u}}(t), \bar{u}(t), t\right)\eta(t)\Big.\Big.\\
&\Big.\Big.+l_z\left(X^{\bar{u}}(t), Y^{\bar{u}}(t), Z^{\bar{u}}(t), \bar{u}(t), t\right)\zeta(t)+l_v\left(X^{\bar{u}}(t), Y^{\bar{u}}(t), Z^{\bar{u}}(t), \bar{u}(t), t\right)v(t)\Big.\Big.\\
&\Big.\Big.-\frac{\bar{h}(v(t), t) \tilde{\beta}\left(\tau^{\bar{u}}\right)}{h^{\bar{u}}\left(\tau^{\bar{u}}\right)}
-\mathbb{E}[l\left(X^{\bar{u}}(\tau^{\bar{u}}), Y^{\bar{u}}(\tau^{\bar{u}}), Z^{\bar{u}}(\tau^{\bar{u}}), \bar{u}(\tau^{\bar{u}}), \tau^{\bar{u}}\right)] \frac{\bar{h}(v(t), t)}{h^{\bar{u}}\left(\tau^{\bar{u}}\right)}\Big) dt\Big]\\
&+\mathbb{E}\left[\beta_x\left(X^{\bar{u}}\left(\tau^{\bar{u}}\right)\right)^{\top} \xi\left(\tau^{\bar{u}}\right)\right]+\mathbb{E}\left[\gamma_y(Y^{\bar{u}}\left(0\right))^{\top} \eta(0)\right]\ge 0.
\end{split}
\end{align}

Consider the case $\tau^{u_\rho}<\tau^{\bar{u}}$. Similarly, by Lemma \ref{lemma 3.4}, \ref{vari-limit}, \ref{equation_lemma} and Assumptions \ref{assumption_1}, \ref{assumption_2}, we can also obtain inequality (\ref{conditionI}).

\textbf{Secondly, we consider the case (ii)} where $\inf\  \{t: \mathbb{E}[\Phi(X^{\bar{u}}(t))] \geq \alpha, t \in[0, T]\}=T.$ If there exists sequence $\rho_n \to 0$ as $n \to +\infty$ such that
$$
\tau^{u_{\rho_n}}=\inf \left\{t: \mathbb{E}\left[\Phi\left(X^{u_{\rho_n}}(t)\right)\right] \geq \alpha, t \in[0, T]\right\}<T,
$$
we can obtain the inequality (\ref{conditionI}).

If there exists sequence $\rho_n \rightarrow 0$ as $n \rightarrow+\infty$ such that
$$
\inf \left\{t: \mathbb{E}\left[\Phi\left(X^{u_{\rho_n}}(t)\right)\right] \geq \alpha, t \in[0, T]\right\}=+\infty,
$$
similar to equation (\ref{casei}), we can obtain
\begin{align}
& \rho^{-1}[J(u(\cdot)+\rho v(\cdot))-J(u(\cdot))] \nonumber\\
= & \rho^{-1} \mathbb{E}\left[\int_0^T l\left(X^{u_\rho}(t), Y^{u_\rho}(t), Z^{u_\rho}(t), u_\rho(t), t\right) d t-\int_0^T l\left(X^{\bar{u}}(t), Y^{\bar{u}}(t), Z^{u_\rho}(t), \bar{u}(t), t\right) d t\right] \nonumber\\
& +\rho^{-1} \mathbb{E}\left[\beta\left(X^{u_\rho}\left(T\right)\right)-\beta\left(X^{\bar{u}}\left(T\right)\right)\right]+\rho^{-1} \mathbb{E}\left[\gamma\left(Y^{u_\rho}(0)\right)-\gamma\left(Y^{\bar{u}}(0)\right)\right] \label{ba1}.
\end{align}
Using Lemma \ref{equation_lemma}, when $\rho \to 0$, we have
\begin{equation}\label{ca1}
\begin{aligned}
 &\rho^{-1} \mathbb{E}\left[\int_0^T l\left(X^{u_\rho}(t), Y^{u_\rho}(t), Z^{u_\rho}(t), u_\rho(t), t\right) d t-\int_0^T l\left(X^{\bar{u}}(t), Y^{\bar{u}}(t), Z^{u_\rho}(t), \bar{u}(t), t\right) d t\right]\\
\rightarrow &\mathbb{E}\Big[\int _ 0 ^ T \big(l_x\left(X^{\bar{u}}(t), Y^{\bar{u}}(t), Z^{u_\rho}(t), \bar{u}(t), t\right) \xi(t)+l_y\left(X^{\bar{u}}(t), Y^{\bar{u}}(t), Z^{u_\rho}(t), \bar{u}(t), t\right) \eta(t)\\
&\Big.+l_z\left(X^{\bar{u}}(t), Y^{\bar{u}}(t), Z^{u_\rho}(t), \bar{u}(t), t\right) \zeta(t)+l_v\left(X^{\bar{u}}(t), Y^{\bar{u}}(t), Z^{u_\rho}(t), \bar{u}(t), t\right) v(t)\big) d t \Big].
\end{aligned}
\end{equation}
Applying Lemma 3.5 of Yang\cite{yang2020varying} and Lemma \ref{equation_lemma} to equation (\ref{ba1}), we have
\begin{align}\label{ca2}
\rho^{-1} \mathbb{E}\left[\beta\left(X^{u_\rho}(T)\right)-\beta\left(X^{\bar{u}}(T)\right)\right]
\to \mathbb{E}\left[\beta_x\left(X^{\bar{u}}(T)\right)\xi(T)\right]\ \ (\rho \to 0),
\end{align}
\begin{align}\label{ca3}
\rho^{-1} \mathbb{E}\left[\gamma\left(Y^{u_\rho}(0)\right)-\gamma\left(Y^{\bar{u}}(0)\right)\right]
\to \mathbb{E}\left[\gamma_y\left(Y^{\bar{u}}(0)\right)\eta(0)\right] \ \ (\rho \to 0).
\end{align}
Combing equations (\ref{ca1}), (\ref{ca2}) and (\ref{ca3}), we have
\begin{align*}
&\mathbb{E}\Big[\int _ 0 ^ T \left(l_x\left(X^{\bar{u}}(t), Y^{\bar{u}}(t), Z^{\bar{u}}(t), \bar{u}(t), t\right) \xi(t)+l_y\left(X^{\bar{u}}(t), Y^{\bar{u}}(t), Z^{\bar{u}}(t), \bar{u}(t), t\right) \eta(t)\Big.\right.\\
&\left.\Big.+l_z\left(X^{\bar{u}}(t), Y^{\bar{u}}(t), Z^{\bar{u}}(t), \bar{u}(t), t\right) \zeta(t)+l_v\left(X^{\bar{u}}(t), Y^{\bar{u}}(t), Z^{\bar{u}}(t), \bar{u}(t), t\right) v(t)\right) d t \Big]\\
&+\mathbb{E}\left[\beta_x\left(X^{\bar{u}}(T)\right)\xi(T)\right]
+\mathbb{E}\left[\gamma_y\left(Y^{\bar{u}}(0)\right)\eta(0)\right]
\ge 0.
\end{align*}

\textbf{Finally, we consider the case $(iii)$} where $inf\ \{t: \mathbb{E}[\Phi(X^{\bar{u}}(t))] \geq \alpha, t \in[0, T]\}=\emptyset$. Thus, for sufficiently small $\rho$, there exists $\tau^{u_\rho}=\tau^{\bar{u}}=T$. Similar to the second scenario in case $(ii)$, we have
\begin{align*}
&\mathbb{E}\Big[\int _ 0 ^ T \left(l_x\left(X^{\bar{u}}(t), Y^{\bar{u}}(t), Z^{\bar{u}}(t), \bar{u}(t), t\right) \xi(t)+l_y\left(X^{\bar{u}}(t), Y^{\bar{u}}(t), Z^{\bar{u}}(t), \bar{u}(t), t\right) \eta(t)\Big.\right.\\
&\left.\Big.+l_z\left(X^{\bar{u}}(t), Y^{\bar{u}}(t), Z^{\bar{u}}(t), \bar{u}(t), t\right) \zeta(t)+l_v\left(X^{\bar{u}}(t), Y^{\bar{u}}(t), Z^{\bar{u}}(t), \bar{u}(t), t\right) v(t)\right) d t \Big]\\
&+\mathbb{E}\left[\beta_x\left(X^{\bar{u}}(T)\right)\xi(T)\right]
+\mathbb{E}\left[\gamma_y\left(Y^{\bar{u}}(0)\right)\eta(0)\right]
\ge 0.
\end{align*}

This completes the proof.
\end{proof}

\section{Stochastic maximum principle}\label{sectionIV}

We introduce the following $\rho$-moving adjoint equations based on $\tau^{\bar{u}}$ and $\tau^{u_\rho}$:
\begin{align}
&\begin{cases}\label{max-1}
\begin{split}
\begin{aligned}
\noindent&\begin{aligned}
-d p^\rho(t)= &\left[f_x(x(t), u(t)) p^\rho(t)+g_x(x(t), y(t), z(t), u(t), t)q(t)\right.\\
&\left.+\sigma_x(x(t), u(t))k^\rho(t)+l_x(x(t), y(t), z(t), u(t), t)\right] dt-k^\rho(t) dW(t),
\end{aligned}\\
\noindent&\begin{aligned}
p^\rho(\tau^{u_\rho} \wedge {\tau^{\bar{u}}})= & \beta_x(x(\tau^{\bar{u}}\wedge \tau^{u_\rho}))-
q\left(\tau^{\bar{u}}\wedge\tau^{u_\rho}\right) \Psi_x\left(x\left(\tau^{\bar{u}}\wedge\tau^{u_\rho}\right)\right)
\end{aligned}
\end{aligned}
\end{split}
\end{cases}\\
&\begin{cases}\label{max-2}
\begin{split}
\begin{aligned}
\noindent&\begin{aligned}
-d q(t)= &{\left[g_y(x(t), y(t), z(t), u(t), t) q(t)+l_y(x(t), y(t), z(t), u(t), t)\right] d t } \\
&+\left[g_z(x(t), y(t), z(t), u(t), t) q(t)+l_z(x(t), y(t), z(t), u(t), t)\right] d W(t)
\end{aligned}\\
\noindent&\begin{aligned}
& q(0)=-\gamma_y(y(0))
\end{aligned}
\end{aligned}
\end{split}
\end{cases}\\
&\begin{cases}\label{max-3}
\begin{aligned}
-d p(t)= &\left[f_x(x(t), u(t)) p(t)+g_x(x(t), y(t), z(t), u(t), t)q(t)\right.\\
&\left.+\sigma_x(x(t), u(t))k(t)+l_x(x(t), y(t), z(t), u(t), t)\right] dt-k(t) dW(t),\\
p(\tau^{\bar{u}})= & \beta_x(x(\tau^{\bar{u}}))-
q\left(\tau^{\bar{u}}\right) \Psi_x\left(x\left(\tau^{\bar{u}}\right)\right)
\end{aligned}
\end{cases}
\end{align}
The above equations are classical FBSDEs. Based on Assumptions \ref{assumption_1} and \ref{assumption_2}, there exist unique solution
$$
(p^\rho(\cdot), k^\rho(\cdot), q(\cdot)) \in \mathcal{M} ^2\left(\mathbb{R}^n\right) \times \mathcal{M} ^2\left(\mathbb{R}^{n\times d}\right) \times \mathcal{M}^2\left(\mathbb{R}^m\right)
$$
which satisfy equations (\ref{max-1}) and (\ref{max-2}). Similarly, there exist unique solution
$$
(p(\cdot), k(\cdot), q(\cdot)) \in \mathcal{M} ^2\left(\mathbb{R}^n\right) \times \mathcal{M} ^2\left(\mathbb{R}^{n\times d}\right) \times \mathcal{M}^2\left(\mathbb{R}^m\right)
$$
which satisfy equations (\ref{max-2}) and (\ref{max-3}).

Then we denote the Hamiltonian and $\rho$-moving Hamiltonian as follows:
\begin{equation}
\begin{aligned}
&H(x, y, z, u, p, k, q, t)= (p, f(x, u))+(k, \sigma(x, u))+(q, g(x, y, z, u, t))+l(x, y, z, u, t)\\
&H^\rho(x, y, z, u, p^\rho, k^\rho, q, t)= (p^\rho, f(x, u))+(k^\rho, \sigma(x, u))+(q, g(x, y, z, u, t))+l(x, y, z, u, t)
\end{aligned}
\end{equation}
where $H^\rho,\ H:(x, y, z, u, p, k, q, t) \in \mathbb{R}^n \times \mathbb{R}^m \times \mathbb{R}^{m\times d} \times U \times \mathbb{R}^n \times \mathbb{R}^{n\times d} \times \mathbb{R}^m  \times[0, T] \rightarrow \mathbb{R}.$

We rewrite the equations (\ref{max-1}), (\ref{max-2}) and (\ref{max-3}) as follows:
\begin{align}
&\begin{cases}\label{hami-1}
\begin{split}
\begin{aligned}
\noindent&\begin{aligned}
-d p^\rho(t)=H^\rho_x(x(t), y(t), z(t), u(t), p^\rho(t), k^\rho(t), q(t), t)dt-k^\rho(t) dW(t),
\end{aligned}\\
\noindent&\begin{aligned}
p^\rho(\tau^{u_\rho} \wedge {\tau^{\bar{u}}})= &  \beta_x(x(\tau^{\bar{u}}\wedge \tau^{u_\rho}))-
q\left(\tau^{\bar{u}}\wedge\tau^{u_\rho}\right) \Psi_x\left(x\left(\tau^{\bar{u}}\wedge\tau^{u_\rho}\right)\right)
\end{aligned}
\end{aligned}
\end{split}
\end{cases}\\
&\begin{cases}\label{hami-2}
\begin{split}
\begin{aligned}
\noindent&\begin{aligned}
-d q(t)= &H_y(x(t), y(t), z(t), u(t), p(t), k(t), q(t), t) dt \\
&+H_z(x(t), y(t), z(t), u(t), p(t), k(t), q(t), t) d W(t)
\end{aligned}\\
\noindent&\begin{aligned}
& q(0)=-\gamma_y(y(0))
\end{aligned}
\end{aligned}
\end{split}
\end{cases}\\
&\begin{cases}\label{hami-3}
\begin{split}
\begin{aligned}
\noindent&\begin{aligned}
-d p(t)=H_x(x(t), y(t), z(t), u(t), p(t), k(t), q(t), t)dt-k(t) dW(t),
\end{aligned}\\
\noindent&\begin{aligned}
p(\tau^{\bar{u}})= & \beta_x(x(\tau^{\bar{u}}))-
q\left(\tau^{\bar{u}}\right) \Psi_x\left(x\left(\tau^{\bar{u}}\right)\right)
\end{aligned}
\end{aligned}
\end{split}
\end{cases}
\end{align}

The following Lemma is necessary in the proof of the stochastic maximum principle.
\begin{lemma}\label{hami-limit}
Let Assumptions \ref{assumption_1}, \ref{assumption_2} and \ref{assumption_3} hold. Suppose that $\tau^{\bar{u}}$ and $\tau^{u_{\rho}}$ is defined in equation (\ref{varying_time}). Then we have
\begin{align*}
&\lim_{\rho \to 0} \mathbb{E}\left [ \left |p^{\rho}\left(t\right)-p\left(t\right) \right |^2  \right ] =0,\\
&\lim_{\rho \to 0} \mathbb{E}\left [\int_{0}^{\tau^{\bar{u}}\wedge\tau^{u_\rho}}
 \left |k^{\rho}\left(t\right)-k\left(t\right) \right |^2 dt \right ] =0
\end{align*}
for $t \in [0, \tau^{u_\rho} \wedge \tau^{\bar{u}}]$.
\end{lemma}

\begin{remark}
Similar to the proof in Lemma \ref{vari-limit}, we can prove Lemma \ref{hami-limit}.
In Lemma \ref{vari-limit} and \ref{hami-limit}, $(\eta^\rho(t), \zeta^\rho(t), p^\rho(t), \kappa^\rho(t))$ can converge to $(\eta(t), \zeta(t), p(t), \kappa(t))$ for $t \in [0, \tau^{\bar{u}} \wedge \tau^{u_\rho}]$. This means that the terminal time can be regarded as $\tau^{\bar{u}}$ in stochastic maximum principle, because $\tau^{\bar{u}}\wedge\tau^{u_\rho} \to \tau^{\bar{u}}$ when $\rho \to 0$ by Lemma \ref{lemma 3.2}.
\end{remark}

\noindent  Now, we present \textbf{the proof of Theorem \ref{Theorem}}:
\begin{proof}
\textbf{ Firstly, we consider the case $(i)$} where $\tau^{\bar{u}}<T$. Applying It\^{o}'s formula to $\xi(t)p^\rho(t) + \eta^\rho(t)q(t)$, one can obtain that
\begin{align*}
&\mathbb{E}[\left(\xi(\tau^{\bar{u}} \wedge \tau^{u_\rho})p^\rho(\tau^{\bar{u}} \wedge \tau^{u_\rho}) + \eta^\rho(\tau^{\bar{u}} \wedge \tau^{u_\rho})q(\tau^{\bar{u}} \wedge \tau^{u_\rho})\right)-\left(\xi(0)p^\rho(0) + \eta^\rho(0)q(0)\right)]\\
=&\mathbb{E}\int_{0}^{\tau^{\bar{u}} \wedge \tau^{u_\rho}}\big[-\Big(l_x(X^{\bar{u}}(t),Y^{\bar{u}}(t),Z^{\bar{u}}(t),\bar{u}(t),t)\xi(t)
+l_y(X^{\bar{u}}(t),Y^{\bar{u}}(t),Z^{\bar{u}}(t),\bar{u}(t),t)\eta^\rho(t)\\
&+l_z(X^{\bar{u}}(t),Y^{\bar{u}}(t),Z^{\bar{u}}(t),\bar{u}(t),t)\zeta^\rho(t)
+l_u(X^{\bar{u}}(t),Y^{\bar{u}}(t),Z^{\bar{u}}(t),\bar{u}(t),t)v(t)\Big)\\
&+H_u^\rho(X^{\bar{u}}(t),Y^{\bar{u}}(t),Z^{\bar{u}}(t),\bar{u}(t),p^\rho(t),k^\rho(t),q(t),t)v(t)\big]dt.
\end{align*}
According to Lemma \ref{lemma 3.2}, \ref{vari-limit} and \ref{hami-limit}, let $\rho \to 0$ on both sides of the above equation and we have
\begin{align}\label{vari}
\begin{split}
&\mathbb{E}\int_{0}^{\tau^{\bar{u}}}H_u(X^{\bar{u}}(t),Y^{\bar{u}}(t),Z^{\bar{u}}(t),\bar{u}(t),p(t),k(t),q(t),t)v(t)dt\\
=&\mathbb{E}[\Big(\xi(\tau^{\bar{u}})\beta_x\left(X^{\bar{u} }\left(\tau^{\bar{u}}\right)\right) -\xi(\tau^{\bar{u}})
q\left(\tau^{\bar{u}}\right) \Psi_x\left(X^{\bar{u}}\left(\tau^{\bar{u}}\right)\right)\Big)
+\kappa\left(\tau^{\bar{u}}\right)q(\tau^{\bar{u}})+\eta(0)\gamma_y(Y^{\bar{u}}(0))]\\
&+\mathbb{E}\int_{0}^{\tau^{\bar{u}}}(l_x(X^{\bar{u}}(t),Y^{\bar{u}}(t),Z^{\bar{u}}(t),\bar{u}(t),t)\xi(t)
+ l_y(X^{\bar{u}}(t),Y^{\bar{u}}(t),Z^{\bar{u}}(t),\bar{u}(t),t)\eta(t)\\
&+ l_z(X^{\bar{u}}(t),Y^{\bar{u}}(t),Z^{\bar{u}}(t),\bar{u}(t),t)\zeta(t)
+ l_u(X^{\bar{u}}(t),Y^{\bar{u}}(t),Z^{\bar{u}}(t),\bar{u}(t),t)v(t))dt.
\end{split}
\end{align}
Applying Lemma \ref{equation_lemma} and \ref{lemma 3.7} to equation (\ref{vari}), one can obtain
\begin{align*}
&\mathbb{E}\int_{0}^{\tau^{\bar{u}}}\Big(H_u(X^{\bar{u}}(t),Y^{\bar{u}}(t),Z^{\bar{u}}(t),\bar{u}(t),p(t),k(t),q(t),t)v(t)+q(\tau^{\bar{u}})\frac{\tilde{\Psi}^{\bar{u}}\left(\tau^{\bar{u}}\right) \bar{h}(v(t), t)}{h^{\bar{u}}\left(\tau^{\bar{u}}\right)}\\
&-q(\tau^{\bar{u}})g\left(X^{\bar{u}}\left(\tau^{\bar{u}}\right), Y^{\bar{u}}\left(\tau^{\bar{u}}\right), Z^{\bar{u}}\left(\tau^{\bar{u}}\right), \bar{u}\left(\tau^{\bar{u}}\right)\right)\frac{\bar{h}(v(t), t)}{h^{\bar{u}}\left(\tau^{\bar{u}}\right)}
-\frac{\bar{h}(v(t), t) \tilde{\beta}^{\bar{u}}\left(\tau^{\bar{u}}\right)}{h^{\bar{u}}\left(\tau^{\bar{u}}\right)}\\
&-\mathbb{E}\left[l\left(X^{\bar{u}}\left(\tau^{\bar{u}}\right), Y^{\bar{u}}\left(\tau^{\bar{u}}\right), Z^{\bar{u}}\left(\tau^{\bar{u}}\right), \bar{u}\left(\tau^{\bar{u}}\right), \tau^{\bar{u}}\right)\right] \frac{\bar{h}(v(t), t)}{h^{\bar{u}}\left(\tau^{\bar{u}}\right)}
\Big)dt \ge 0.
\end{align*}
Therefore we set $v(t)=u(t)-\bar{u}(t)$ where $u(\cdot) \in \mathcal{U}$. When $\rho \to 0$, we can obtain
\begin{equation}
\label{theorem1}
\begin{aligned}
&H_u(X^{\bar{u}}(t),Y^{\bar{u}}(t),Z^{\bar{u}}(t),\bar{u}(t),p(t),k(t),q(t),t)(u-\bar{u}(t))+q(\tau^{\bar{u}})\frac{\tilde{\Psi}^{\bar{u}}\left(\tau^{\bar{u}}\right) \bar{h}(u-\bar{u}(t), t)}{h^{\bar{u}}\left(\tau^{\bar{u}}\right)}\\
&-q(\tau^{\bar{u}})g\left(X^{\bar{u}}\left(\tau^{\bar{u}}\right), Y^{\bar{u}}\left(\tau^{\bar{u}}\right), Z^{\bar{u}}\left(\tau^{\bar{u}}\right), \bar{u}\left(\tau^{\bar{u}}\right)\right)\frac{\bar{h}(u-\bar{u}(t), t)}{h^{\bar{u}}\left(\tau^{\bar{u}}\right)}
-\frac{\bar{h}(u-\bar{u}(t), t) \tilde{\beta}^{\bar{u}}\left(\tau^{\bar{u}}\right)}{h^{\bar{u}}\left(\tau^{\bar{u}}\right)}\\
&-\mathbb{E}\left[l\left(X^{\bar{u}}\left(\tau^{\bar{u}}\right), Y^{\bar{u}}\left(\tau^{\bar{u}}\right), Z^{\bar{u}}\left(\tau^{\bar{u}}\right), \bar{u}\left(\tau^{\bar{u}}\right), \tau^{\bar{u}}\right)\right] \frac{\bar{h}(u-\bar{u}(t), t)}{h^{\bar{u}}\left(\tau^{\bar{u}}\right)} \ge 0
\end{aligned}
\end{equation}
for any $u \in U$, a.e. $t \in\left[0, \tau^{\bar{u}}\right)$, and $P-$ a.s.

 \textbf{Secondly, we consider the case $(ii)$} where $\inf\left\{t: \mathbb{E}[\Phi(X^{\bar{u}}(t))] \geq \alpha, t \in[0, T]\right\}=T$. If there exists sequence $\rho_n \to 0$ as $n \to +\infty$ such that
$$
\tau^{u_{\rho_n}}=\inf \left\{t: \mathbb{E}\left[\Phi\left(X^{u_{\rho_n}}(t)\right)\right] \geq \alpha, t \in[0, T]\right\}<T,
$$
we can deduce the inequality (\ref{theorem1}).

If there exists sequence $\rho_n \to 0$ as $n \to +\infty$ such that
$$
\tau^{u_{\rho_n}}=\inf \left\{t: \mathbb{E}\left[\Phi\left(X^{u_{\rho_n}}(t)\right)\right] \geq \alpha, t \in[0, T]\right\}=T.
$$
By Lemma \ref{equation_lemma} and \ref{lemma 3.7}, we have
\begin{align*}
\begin{split}
&\mathbb{E}\int_{0}^{T}H_u(X^{\bar{u}}(t),Y^{\bar{u}}(t),Z^{\bar{u}}(t),\bar{u}(t),p(t),k(t),q(t),t)v(t)dt\\
=&\mathbb{E}[\Big(\xi(T)\beta_x\left(X^{\bar{u} }\left(T\right)\right) -\xi(T)
q\left(T\right) \Psi_x\left(X^{\bar{u}}\left(T\right)\right)\Big)
+\kappa\left(T\right)q(T)+\gamma_y(Y^{\bar{u}}(0))\eta(0)]\\
&+\mathbb{E}\int_{0}^{T}(l_x(X^{\bar{u}}(t),Y^{\bar{u}}(t),Z^{\bar{u}}(t),\bar{u}(t),t)\xi(t)
+ l_y(X^{\bar{u}}(t),Y^{\bar{u}}(t),Z^{\bar{u}}(t),\bar{u}(t),t)\eta(t)\\
&+ l_z(X^{\bar{u}}(t),Y^{\bar{u}}(t),Z^{\bar{u}}(t),\bar{u}(t),t)\zeta(t)
+ l_u(X^{\bar{u}}(t),Y^{\bar{u}}(t),Z^{\bar{u}}(t),\bar{u}(t),t)v(t))dt.
\end{split}
\end{align*}
Similar to the proof in the case $(i)$, we can obtain that
\begin{align*}
&\mathbb{E}\int_{0}^{T}[H_u(X^{\bar{u}}(t),Y^{\bar{u}}(t),Z^{\bar{u}}(t),\bar{u}(t),p(t),k(t),q(t),t)v(t)]dt \ge 0.
\end{align*}
Therefore we set $v(t)=u(t)-\bar{u}(t)$, and one can obtain
\begin{equation*}
H_u(X^{\bar{u}}(t),Y^{\bar{u}}(t),Z^{\bar{u}}(t),\bar{u}(t),p(t),k(t),q(t),t)(u(t)-\bar{u}(t)) \ge 0
\end{equation*}
where $t \in [0,T]$.for any $u \in U$, a.e. $t \in\left[0, T\right)$, and $P-$ a.s.
\textbf{Finally, we consider the case $(iii)$} where $\left\{t: \mathbb{E}[\Phi(X^{\bar{u}}(t))] \geq \alpha, t \in[0, T]\right\}=\emptyset$. In this case, $\tau^{u_\rho}=\tau^{\bar{u}}=T$, and we can obtain the same results as in the case $(ii)$.

This completes the proof.
\end{proof}

In case $(iii)$, $\left\{t: \mathbb{E}[\Phi(X^{\bar{u}}(t))] \geq \alpha, t \in[0, T]\right\}=\emptyset$, it means that $\mathbb{E}[\Phi(X^{\bar{u}}(t))] < \alpha, t \in[0, T]$. Thus it is same with the classical optimal control problem without state constraints. In the following, we give an example to verify the results of Theorem \ref{Theorem}.

\begin{example}
Let $n=m=d=1$, $T=1$, $U=[1,2]$, $\alpha=1$, $f(x,u)=x+u$, $\sigma(x,u)=0$, $g(x,y,z,u)=x+y+u$, $\Psi(x)=0$, $\beta(x)=0$, $\gamma(x)=0$, $l(x,y,z,u,t)=u$, $\Phi(x)=x$ and $X_0 = 0$. Therefore, the solution $(X^u\left ( \cdot \right ),  Y^u\left ( \cdot \right ),  Z^u\left ( \cdot \right ))$ of FBSDEs (\ref{2.4}) is given as follows
\begin{equation}\label{e1}
\begin{cases}
 X^u\left ( t \right )=\int_0^tu(s)e^{t-s}ds,
 \\Y^u\left ( t \right )=-\int_t^{\tau^u} (x(s)+u(s))e^{t-s}ds,
 \\Z^u\left ( t \right )= 0,
\end{cases}
\end{equation}
where the terminal time $\tau^u$ is defined as follows
\begin{equation}\label{exterminal}
    \tau ^u=\inf\left \{ t:\mathbb{E}\left [ X^u\left ( t
 \right )  \right ] \ge 1 , t\in\left [ 0,1\right ]  \right \} \wedge 1.
\end{equation}
The unique solution $(p\left ( \cdot \right ),  k\left ( \cdot \right ), q\left ( \cdot \right ))$ of equations (\ref{max-2}) and (\ref{max-3}) are
\begin{equation}\label{e2}
 p\left ( t \right )= 0,\ k\left ( t \right )= 0,\ q\left ( t \right )= 0,\ 0\leq t\leq 1.
\end{equation}
Note that $
h^u\left ( t \right ) =\mathbb{E}\left [X^u\left ( t \right ) + u\left ( t \right )\right ]=
X^u\left ( t \right ) + u\left ( t \right )$, we have
\begin{equation*}
\bar{h}(v(t), t)=\lim _{\rho \rightarrow 0} \frac{h^{u_\rho}(t)-h^{\bar{u}}(t)}{\rho}
=v(t)+\lim_{\rho \rightarrow 0} \frac{X^{u_\rho}(t)-X^{\bar{u}}(t)}{\rho}
=v(t)+\int_0^t e^{t-s} v(s)ds.
\end{equation*}

Now, we suppose $\tau^{\bar{u}}<1$ and the cost functional is
$$
J(u(\cdot))=\int_0^{\tau^{\bar{u}}}\bar{u}(s) \mathrm{d} s.
$$
From $(i)$ of Theorem \ref{Theorem}, combing equations (\ref{e1}) and (\ref{e2}), we have
\begin{equation*}
\begin{aligned}
&H_u(X^{\bar{u}}(t),Y^{\bar{u}}(t),Z^{\bar{u}}(t),\bar{u}(t),p(t),k(t),q(t),t)(u-\bar{u}(t))+q(\tau^{\bar{u}})\frac{\tilde{\Psi}^{\bar{u}}\left(\tau^{\bar{u}}\right) \bar{h}(u-\bar{u}(t), t)}{h^{\bar{u}}\left(\tau^{\bar{u}}\right)}\Big.\\
&-q(\tau^{\bar{u}})g\left(X^{\bar{u}}\left(\tau^{\bar{u}}\right), Y^{\bar{u}}\left(\tau^{\bar{u}}\right), Z^{\bar{u}}\left(\tau^{\bar{u}}\right), \bar{u}\left(\tau^{\bar{u}}\right)\right)\frac{\bar{h}(u-\bar{u}(t), t)}{h^{\bar{u}}\left(\tau^{\bar{u}}\right)}
-\frac{\bar{h}(u-\bar{u}(t), t) \tilde{\beta}^{\bar{u}}\left(\tau^{\bar{u}}\right)}{h^{\bar{u}}\left(\tau^{\bar{u}}\right)}\Big.\\
&\Big.-\mathbb{E}\left[l\left(X^{\bar{u}}\left(\tau^{\bar{u}}\right), Y^{\bar{u}}\left(\tau^{\bar{u}}\right), Z^{\bar{u}}\left(\tau^{\bar{u}}\right), \bar{u}\left(\tau^{\bar{u}}\right), \tau^{\bar{u}}\right)\right] \frac{\bar{h}(u-\bar{u}(t), t)}{h^{\bar{u}}\left(\tau^{\bar{u}}\right)}\\
=&(u(t)-\bar{u}(t)) - \bar{u}(\tau^{\bar{u}})
\frac{1}{1+\bar{u}(\tau^{\bar{u}})}\left ( \int_0^te^{t-s}(u(s)-\bar{u}(s))ds+u(t)-\bar{u}(t)  \right ) \ge 0
\end{aligned}
\end{equation*}
where $X^{\bar{u}}(\tau^{\bar{u}})=1$ and $Y^{\bar{u}}(\tau^{\bar{u}})=0$. Then it follows that
\begin{equation}\label{excon}
\frac{1}{1+\bar{u}\left(\tau^{\bar{u}}\right)}(u(t)-\bar{u}(t))-
\bar{u}(\tau^{\bar{u}})
\frac{1}{1+\bar{u}(\tau^{\bar{u}})} \int_0^te^{t-s}(u(s)-\bar{u}(s))ds\ge 0
\end{equation}
for $t \in [0, \tau^{\bar{u}}]$.
Then, it follows that
\begin{equation*}
u(t)-\bar{u}(t)\int_0^te^{t-s}u(s)ds
\ge \bar{u}(t)-\bar{u}(t)\int_0^te^{t-s}\bar{u}(s)ds
\end{equation*}
Therefore we can obtain an optimal control $\bar{u}(t)=1$ for
$t \ a.e. \in [0,\tau^{\bar{u}}]$ and
\begin{equation*}
\begin{aligned}
&X^{\bar{u}}(t)=e^t-1,\\
&Y^{\bar{u}}(t)=-e^t(\tau^{\bar{u}}-t).
\end{aligned}
\end{equation*}

Then, by equation (\ref{varying_time}), we have $\tau^{\bar{u}} = ln2$ and optimal pair $(\bar{u}(\cdot), \tau^{\bar{u}})$ satisfies the equation (\ref{excon}). The cost functional is
$$
J(\bar{u}(\cdot))=\int_0^{\tau^{\bar{u}}}\bar{u}(s)\mathrm{d} s
=ln2.
$$

In the following, we consider the classical stochastic optimal control problems under state constraints. The cost functional is given as follows:
$$
\tilde{J}(u(\cdot))=\int_0^{T}u(s)\mathrm{d} s,
$$
and the state constraints is
$
\mathbb{E}\left [ X^u\left ( T\right )  \right ] \ge 1.
$
By equation (\ref{e1}), we have that the optimal control is $\bar{u}(t)=1$ for $t \in [0,T]$. The cost functional is
$$
\tilde{J}(\bar{u}(\cdot))=\int_0^{T}\bar{u}(s)\mathrm{d} s=1.
$$

The above results show that we can stop the controlled systems at terminal time $\tau^{\bar{u}}=ln2$ where $X^{\bar{u}}\left ( \tau^{\bar{u}}\right )=1$ and $ ln2=J(\bar{u}(\cdot))<\tilde{J}(\bar{u}(\cdot))=1$, and we can obtain a better performance cost functional.
\end{example}
\section{Conclusion}\label{sectionV}
As shown in \cite{yang2020varying} and \cite{yang2022varying}, the terminal time varies according to the state constraints can obtain a better performance cost functional than the classical optimal control problem. Here, we extend the model in \cite{yang2020varying} to a new stochastic optimal control problem driven by FBSDEs. To solve this new recursive stochastic optimal control problems, we introduce novel $\rho$-moving variational and  adjoint equations and establish the related stochastic maximum principle.

However, future studies should be considered, such as establishing linear quadratic optimal control problem with varying terminal time, mean-field optimal control problem with varying terminal time, and so on.

\bibliographystyle{unsrt}
\bibliography{citation}

\end{document}